\documentclass[12pt,a4paper]{article}
\usepackage{amsfonts}
\usepackage{amsmath, amsthm}
\usepackage{latexsym}
\usepackage{amssymb}
\usepackage{amscd}
\usepackage{verbatim}

\makeatletter
\@addtoreset{equation}{section}
\makeatother

\newtheorem{theorem}{Theorem}[section]
\newtheorem{lemma}[theorem]{Lemma}
\newtheorem{corollary}[theorem]{Corollary}

\theoremstyle{definition}
\newtheorem{definition}[theorem]{Definition}
\newtheorem{example}[theorem]{Example}
\newtheorem{remark}[theorem]{Remark}

\begin{document}

\title{On constructions and parameters of symmetric configurations $v_{k}$}
\author{Alexander A. Davydov$^{1}$, Giorgio Faina$^{2}$, Massimo Giulietti$^{2}$, \\
Stefano Marcugini$^{2}$, Fernanda Pambianco$^{2}$}
\date{}\maketitle

$^{1}$Institute for Information Transmission Problems, Russian
Academy of Sciences, Bol'\-shoi Karetnyi per. 19, GSP-4,
Moscow, 127994, Russian Federation, E-mail:
adav@iitp.ru\medskip

$^{2}$Dipartimento di Matematica e Informatica, Universit\`{a}
degli Studi di Perugia, Via Vanvitelli 1, Perugia, 06123,
Italy, E-mail: faina@dmi.unipg.it; giuliet@dmi.unipg.it;\\
gino@dmi.unipg.it; fernanda@dmi.unipg.it\medskip

\noindent \textbf{Abstract: The spectrum of possible parameters
of symmetric configurations is investigated. We both survey
known constructions and results, and propose some new
construction methods. Many new parameters are obtained, in
particular for cyclic symmetric configurations, which are
equivalent  to deficient cyclic difference sets. Both Golomb
rulers and modular Golomb rulers are a key tool in our
investigation. Several new upper bounds on the minimum integer
$E(k)$ such that for each $v\geq E(k)$ there exists a symmetric
configuration $v_{k}$ are obtained. Upper bounds of the same
type are provided for cyclic symmetric configurations. From the
standpoint of applications, it should be noted that our results
extend the range of possible parameters of LDPC codes,
generalized LDPC codes, and quasi-cyclic LDPC codes.}\medskip


\section{Introduction}

Configurations are interesting combinatorial structures. They were defined
in 1876. For an introduction to the problems connected with the
configurations and their history, see \cite{Gropp-ConfGeomCombin,Gropp-Handb,Grunbaum} and the references therein.

\begin{definition}
\label{Def1-Configur}\cite{Gropp-Handb}

\begin{description}
\item[(i)] A configuration $(v_{r},b_{k})$ is an incidence structure of $v$
points and $b$ lines such that each line contains $k$ points, each point
lies on $r$ lines, and two distinct points are connected by \emph{at most}
one line.

\item[(ii)] If $v=b$ and, hence, $r=k$, the configuration
    is \emph{symmetric}, and it is referred to as a
    configuration $v_{k}$.

\item[(iii)] The \emph{deficiency} $d$ of a configuration $(v_{r},b_{k})$ is
the value $d=v-r(k-1)-1$.
\end{description}
\end{definition}

A symmetric configuration $v_k$ is  \emph{cyclic} if there
exists a permutation of the set of its points mappings blocks
to blocks, and acting regularly on both points and blocks.
Equivalently, $v_k$ is cyclic if one of its incidence matrix is
circulant.

Steiner systems are configurations with $d=0$ \cite{Gropp-Handb}. The
deficiency of a symmetric configuration $v_{k}$ is $d=v-(k^{2}-k+1)$. The
deficiency of $v_k$ is zero if and only if  $v_{k}$ is a finite
projective plane of order $k-1$. In general, $d$ indicates the number of
points not joined with an arbitrary point or the number of lines parallel to
an arbitrary line, see \cite{AFLN-ConfigGraphs,FunkLabNap,Gropp-nk,Gropp-Handb}.

A configuration $(v_{r},b_{k})$ can be treated also as a
$k$-uniform $r$ -regular linear hypergraph with $v$ vertices
and $b$ hyperedges \cite{Gropp-ConfGraph,Gropp-Handb}.
Connections of configurations $(v_{r},b_{k})$ with numerical
semigroups are noted in \cite{BrasAmorosSemigroup}. Some
analogies between configurations $(v_{r},b_{k})$, regular
graphs, and molecule models of chemical elements are remarked
in \cite{Gropp-Chemic}. As an example of a practical applying
configurations (symmetric and nonsymmetric) we mention also the
problem of user privacy for using database, see
\cite{DomingoAmorosPeertoPeer,StokesPeertoPeer} and the
references therein.

Denote by $\mathbf{M}(v,k)$ an incidence matrix of a symmetric
configuration $v_{k}.$ Any matrix $\mathbf{M}(v,k)$ is a
$v\times v$ 01-matrix with $k$ units in every row and column;
moreover, the $2\times 2$ matrix $\mathbf{J}_{2}$ consisting of
all units is not a submatrix of $\mathbf{M}(v,k)$. Therefore, $
\mathbf{M}(v,k)$ is a $\mathbf{J}_{2}$\emph{-free matrix}. Two
incidence matrices of the same configuration may differ by a
permutation on the rows and the columns.

A matrix $\mathbf{M}(v,k)$ can also be considered as a biadjacency matrix of
a $k$-regular bipartite graph without multiple edges. The biadjacency matrix
describes connections of two vertex subsets of the graph so that the
adjacency matrix has the form
\begin{equation*}
\left[ {{
\begin{array}{cc}
\mathbf{0}_{v} & \mathbf{M}(v,k) \\
\mathbf{M}^{tr}(v,k) & \mathbf{0}_{v}
\end{array}
}}\right]
\end{equation*}
where $tr$ stands for transposition, and
 $\mathbf{0}_{v}$ denotes the zero $v\times v$ matrix. This graph is the \emph{Levi graph}
of the configuration $v_{k}$ \cite[Sec.\thinspace
7.2]{Gropp-Handb}. As $ \mathbf{M}(v,k)$ is a
$\mathbf{J}_{2}$-free,\emph{\ }the graph has girth at
least~six, i.e. it does not contain 4-cycles. Such graphs are
useful for the construction of bipartite-graph codes that can
be treated as\emph{\ low-density parity-check} (LDPC)
\emph{codes} or \emph{generalized }LDPC codes
\cite{AfDaZ}--\cite{BargZemor},\cite{DGMP-ACCT2008,DGMP-GraphCodes,GabidISIT,QC
Encoder,MilFos2008}. If $\mathbf{M}(v,k)$ is\emph{\ circulant},
then the corresponding LDPC code is \emph{quasi-cyclic}; it can
be encoded with the help of shift-registers with relatively
small complexity, see
\cite{DGMP-ACCT2008,DGMP-GraphCodes,GabidISIT,QC Encoder} and
the references therein.

Matrices $\mathbf{M}(v,k)$ consisting of square circulant
submatrices have a number of useful properties, e.g. they are
more suitable for LDPC codes implementation. We say that a
01-matrix\emph{\ }$\mathbf{A}$ is {\em block double-circulant}
(BDC for short) if $\mathbf{A }$ consists of square circulant
blocks whose weights  give rise to a circulant matrix (see
Definition \ref{def3.1_block double-circul} for details). A
configuration $v_{k}$ with a BDC incidence matrix $\mathbf{M}
(v,k) $ is called a \emph{BDC symmetric configuration.
}Symmetric and non-symmetric configurations with incidence
matrices consisting of square circulant blocks are considered,
e.g. in \cite{DGMP-ACCT2008}--\cite
{DGMP-GraphCodes},\cite{Pepe}.

Cyclic configurations are considered, for instance, in
\cite{DGMP-ACCT2008}--\cite
{DGMP-GraphCodes},\cite{Funk2008,Gropp-nk,Lipman,MePaWolk}. A
standard method to construct cyclic configurations (or,
equivalently,  circulant matrices $M_{v,k}$) is based on {\em
Golomb rulers}
\cite{Dimit,Funk2008,WolfMathWorld-GR,Gropp-nk},\cite{Shearer-Handb}--\cite
{ShearerWebModulGR}.

\begin{definition}
\label{Def1_GR}\cite{Shearer-Handb,Funk2008}

\begin{description}
\item[(i)] A \emph{Golomb ruler} $\mathrm{G}_{k}$ of \emph{order} $k$ is an
ordered set of $k$ integers $(a_{1},a_{2},\ldots ,a_{k})$ such that $0\leq
a_{1}<a_{2}<\ldots <a_{k}$ and all the differences $\{a_{i}-a_{j}\,|\,1\leq
j<i\leq k\}$ are distinct. The \emph{length} $L_{\mathrm{G}}(k)$ of the
ruler $\mathrm{G}_{k}$ is equal to $a_{k}-a_{1}$.

\item[(ii)] A Golomb ruler $\mathrm{G}_{k}$ is an
    \emph{optimal Golomb ruler} $ \mathrm{OG}_{k}$ if no
    shorter Golomb ruler of the same order $k$ exists. Let
    $L_{\mathrm{OG}}(k)$ and $L_{\overline{\mathrm{G}}}(k)$
    be the length of an optimal ruler $\mathrm{OG}_{k}$ and
of the \emph{shortest known} Golomb ruler $
\overline{\mathrm{G}}_{k}$, respectively.

\item[(iii)] A $(v,k)$ \emph{modular Golomb ruler} is an ordered set of $k$
integers $(a_{1},a_{2},\ldots ,a_{k})$ such that $0\leq a_{1}<a_{2}<\ldots
<a_{k}$ and all the differences $\{a_{i}-a_{j}\,|\,1\leq i,j\leq k;$ $i\neq
j\}$ are distinct and nonzero modulo $v$.

\end{description}
\end{definition}
Clearly, $L_{\overline{\mathrm{G}}}(k)\geq L_{\mathrm{OG}}(k)$ holds.

For any value $\delta \geq 0$, Golomb rulers
$(a_{1},a_{2},\ldots ,a_{k})$ and $ (a_{1}+\delta ,a_{2}+\delta
,\ldots ,a_{k}+\delta )$ have the same properties. Usually,
$a_{1}=0$ is assumed.

\begin{remark}
A $(v,k)$ modular Golomb ruler is also called a \emph{deficient
cyclic difference set} with deficiency $d=v-(k^{2}-k+1)$. For a
deficient cyclic difference set the deficiency $d$ is the
number of elements in $\mathbb{Z} _{v}\backslash \{0\}$ not
represented by any difference $a_{i}-a_{j}$ \cite {Funk2008}.
Note that the expression \textquotedblleft deficient cyclic
difference set\textquotedblright\ is used in \cite{Funk2008},
whereas in \cite{Lipman} and \cite{MePaWolk}  the expressions
\textquotedblleft difference set modulo $v$ \textquotedblright\
and \textquotedblleft deficient difference set in $
\mathbb{Z}_{v}$\textquotedblright\ are adopted.
\end{remark}

\begin{remark}
\label{rem1_DifTriang=GolRulDifPack=ModGolRul} Golomb rulers
and modular Golomb rulers are deeply connected with difference
triangle sets and difference packings, see e.g.
\cite{LingDTSAP,Shearer-Handb,ShearerReport}. In particular,
according to the notation of \cite{Shearer-Handb},  a Golomb
ruler $\mathrm{G}_{k}$ is a difference triangle set $(1,k-1)$-$
\mathrm{D\Delta S}$, whereas a $(v,k)$ modular Golomb ruler is
a difference packing $1$-$\mathrm{DP}(v,k)$ \cite[Prop. 19.9,
Rem. 19.24]{Shearer-Handb}. If $a_{1}=0$ the corresponding
object is said to be \emph{normalized}. Note also that in
\cite{Swanson}, the expression \textquotedblleft  planar cyclic
difference packing modulo $v$\textquotedblright\ is used for an
object equivalent to a $(v,k)$ modular Golomb ruler.
\end{remark}

In \cite{Dimit} it is proved that
\begin{equation*}
L_{\mathrm{OG}}(k)>k^{2}-2k\sqrt{k}+\sqrt{k}-2.
\end{equation*}
Currently, the optimal lengths $L_{\mathrm{OG}}(k)$ are known
only for orders $k\leq 25$
\cite{Dimit,WolfMathWorld-GR,Shearer-Handb,ShearerWebShortest}.
So, for $k\le 25$ we have
$L_{\overline{\mathrm{G}}}(k)=L_{\mathrm{OG}}(k)$. The proof of
the optimality of a Golomb ruler is a hard problem needing
exhaustive computer search. On the other hand, for sufficiently
large orders $k$, relatively short Golomb rulers are
constructed and are available online, see e.g. the internet
resources
\cite{Dimit,WolfMathWorld-GR,ShearerWebShortest,ShearerWebModulGR}
and the references therein. For $k\leq 150$, the order of
magnitude of the lengths $L_{\overline{\mathrm{G}}}(k)$ of the
shortest known Golomb rulers is $ck^2$  with $c\in[0.7,0.9]$,
see
\cite{Dimit,Funk2008,WolfMathWorld-GR,Gropp-nk,Gropp-Handb,Shearer-Handb,ShearerWebShortest}.
Moreover, \begin{equation*} \,L_{\mathrm{OG}}(k)\leq
L_{\overline{\mathrm{G}}}(k)<k^{2}\text{ for } k<65000,
\end{equation*}
see \cite{Dimit}. Other constructions for large $k$ can be
found  in \cite{Draka}.

We say that a 0,1-\emph{vector} ${\mathbf
u}=(u_{0},u_{1,}\ldots ,u_{v-1})$ {\em corresponds} to a
(modular) Golomb ruler if the increasing sequence  of integers
$j\in \{0,1,\ldots,v-1\}$ such that $u_j=1$ form a (modular)
Golomb ruler.

Recall that \emph{weight} of a \emph{circulant }$0,1$\emph{-matrix }is the
number of units in each its row.

\begin{theorem}
\label{Th1_2L+1}\emph{\cite[Sec.\thinspace 4]{Gropp-nk}}

\begin{description}
\item[(i)] Any Golomb ruler $\mathrm{G}_{k}$ of length
    $L_{\mathrm{G}}(k)$ is a $(v,k)$ modular Golomb ruler
    for all $v$ such that $v\geq 2L_{\mathrm{G} }(k)+1$.

\item[(ii)] A circulant $v\times v$ \emph{0,1}-matrix of
    weight $k$ is an incidence matrix $\mathbf{M}(v,k)$ of
    a cyclic symmetric configuration $ v_{k} $ if and only
    if the first row of the matrix corresponds to a $(v,k)$
    modular Golomb ruler.
\end{description}
\end{theorem}

We remark that  (ii) of Theorem \ref{Th1_2L+1} is not
explicitly stated in \cite{Gropp-nk}. However, the assertion
can be easily deduced from the results in \cite{Gropp-nk}.

\begin{corollary}
\label{Cor1_GolombBound}\emph{\cite[Sec.\thinspace 4]{Gropp-nk}} For all $v$
such that
\begin{equation}
v\geq 2L_{\overline{\mathrm{G}}}(k)+1,  \label{eq1_GolombBound}
\end{equation}
there exists a cyclic symmetric configuration $v_{k}$.
\end{corollary}

We call the value $G(k)=2L_{\overline{\mathrm{G}}}(k)+1$ the \emph{Golomb
bound}.

It is well known that  $v\geq k^{2}-k+1$ holds  for
configurations $v_{k}$, and that the lower bound is attained if
and only if there exists a projective plane of the order $k-1$
\cite{Gropp-nk,Gropp-Handb}. We call  $P(k)=k^{2}-k+1$ the
\emph{projective plane bound}.

Let $v_{\delta }(k)$ be the smallest possible value of $v$ for
which a $(v,k)$ modular Golomb ruler (or, equivalently, a
cyclic symmetric configuration) exists.

Finally, we introduce other two bounds.  The {\em existence
bound} $E(k)$ is the integer such that for any $v\geq E(k)$,
there exists a symmetric configuration $v_{k}$. Similarly, the
\emph{cyclic existence bound} $E_{c}(k)$ is the integer such
that for any $v\geq E_c(k)$, there exists a cyclic $v_{k}$.

Clearly, for a fixed $k$, we have
\begin{eqnarray}
k^{2}-k+1 &=&P(k)\leq E(k)\leq E_{c}(k)\leq G(k)=2L_{\overline{\mathrm{G}}
}(k)+1.  \label{eq1_bounds} \\
k^{2}-k+1 &=&P(k)\leq v_{\delta }(k)\leq E_{c}(k)\leq G(k)=2L_{\overline{
\mathrm{G}}}(k)+1.  \label{eq1_boundsCyclic}
\end{eqnarray}

The aim of this work is threefold:
\begin{itemize}
\item to survey the vast body of literature on  constructions and parameters of symmetric
configurations $v_{k}$;
\item to describe new construction methods, paying special
    attention to constructions producing {circulant} and {
    block double-circulant} incidence matrices
    $\mathbf{M}(v,k)$;

\item to investigate the\emph{\ spectrum }of possible
    {parameters} of {\ symmetric configurations }$v_{k}$
    (with special attention to {\ parameters} of {cyclic
    symmetric configurations}) in the
interval
\begin{equation}
k^{2}-k+1=P(k)\leq v<G(k)=2L_{\overline{\mathrm{G}}}(k)+1.
\label{eq1_region}
\end{equation}
\end{itemize}

Our main achievements are new constructions of BDC incidence
matrices (see Theorems \ref{th3.2_Asigma-s} and
\ref{th3.3_AfPln},  together the examples in Section
\ref{sec_doub-circ}), improvements on the known upper bounds on
$E(k)$ and $ E_{c}(k)$, and several new parameters for cyclic
and non-cyclic configurations $v_{k}$, see Sections
\ref{sec_ParamCyclicConfig} and~\ref{sec6_spectrum}.

From the stand point of  applications, including Coding Theory,
it is sometimes useful to have different matrices ${\mathbb
M}(v,k)$ for the same $v$ and $k$. This is why we attentively
consider various constructions, even when they provide
configurations with the same parameters.

The \emph{Extension Construction}, as introduced in
\cite{AfDaZ,AfDaZ-InfProc}, plays a key role for investigation
of the spectrum of possible\ parameters of\ symmetric
configurations $v_{k},$ $k\geq 11$, as it provides
\emph{intervals } of values of $v$ for a fixed $k.$ To be
successfully applied, the Extension Construction needs a
convenient starting incidence matrix.  Block double-circulant
matrices turn out to be particularly useful in this context,
see Corollary \ref{cor3.2_wh-in0,1}. In this work we use both
the original starting matrices of \cite{AfDaZ,AfDaZ-InfProc}
and some new ones obtained by our new constructions, see
Example \ref{Ex_4_extension}.

We remark that new cyclic configurations provide new modular
Golomb rulers, i.e. new deficient cyclic difference~sets. Note
also that methods considered in this work could be also used to
construct  non-symmetric configurations $(v_{r},b_{k}).$

The paper is organized as follows. In Section \ref{sec_known},
the known constructions and parameters of configurations
$v_{k}$ are considered. In particular,  a geometrical
construction from \cite{DGMP-ACCT2008,DGMP-GraphCodes}, and the
Extension Construction from \cite{AfDaZ,AfDaZ-InfProc} are
described. In Section \ref{sec_doub-circ}, some new
constructions of block double-circulant incidence matrices
$\mathbf{M}(v,k)$ are proposed. In
Section~\ref{sec_admitExten}, methods for constructing matrices
admitting extensions are described. In Sections
\ref{sec_ParamCyclicConfig} and~\ref{sec6_spectrum}, our
results on the spectra of parameters of cyclic and non-cyclic
configurations are reported. An Appendix contains the proof of
one of  theorems from Section \ref{sec_doub-circ}.

Some results of this work were published without  proofs in
\cite{DGMP-ACCT2008,DGMP-Petersb2009}.

\section{Some known constructions and parameters of configurations $v_{k}$
with $P(k)\leq v<G(k)\label{sec_known}$}

The aim of  this section is to provide a list of pairs $(v,k)$
for which a (cyclic) symmetric configuration $v_k$ is known to
exist, see Equations
\eqref{eq2_cyclicPG(2,q)}-\eqref{eq2_tetaExten}. In most cases
a brief description of the corresponding configuration is
given. Infinite families of configuration $v_{k}$ given in this
section are considered in
\cite{AFLN-graphs}--\cite{AfDaZ-InfProc},\cite{Bose,DGMP-ACCT2008,GH,Lipman,MePaWolk},%
\cite{DGMP-GraphCodes}--\cite{FunkLabNap},\cite{Gropp-nk}--\cite{Gropp-Handb};
see also the references therein.

Throughout the paper, $q$ is a prime power and $p$ is a prime.
Let $F_{q}$ be Galois field of $q$ elements. Let $F_{q}^{\ast
}=F_{q}\backslash \{0\}.$ Let $\mathbf{0}_{u}$ be the zero
$u\times u$ matrix. Denote by $\mathbf{P} _{u}$ a permutation
matrix of order $u.$

First we recall that several pairs $(v,k-\delta)$ can be actually obtained from a given $v_k$.
\begin{theorem}
\label{Th2-Mend} \emph{\cite{MePaWolk}} If a (cyclic) configuration $v_{k}$
exists, then for each $\delta $ with $0\leq \delta <k$ there exists a family
of (cyclic) configurations $v_{k-\delta }$ as well.
\end{theorem}
At once we note that a cyclic configuration $v_{k}$ gives a family of cyclic
configurations $v_{k-\delta }$ obtained by distmissng $\delta $ units in the 1-st row
of its incidence matrix. For the general case, Theorem~\ref{Th2-Mend} is based on the
fact that an incidence matrix $\mathbf{M(}v,k\mathbf{)}$ can be represented
as a sum of $k$ permutations $v\times v$ matrices (in many ways). This fact
follows from the results of Steinits (1894) and K\"{o}nig (1914), see e.g.
\cite[Sec.\thinspace 5.2]{Gropp-ConfGraph} and \cite[Sec. 2.5]{Grunbaum}.

The value $\delta $ appearing in Equations
\eqref{eq2_cyclicPG(2,q)}-\eqref{eq2_tetaExten} is connected
with Theorem \ref{Th2-Mend}. When a reference is given, it
usually refer to the case $\delta =0$.

The families giving rise to pairs
(\ref{eq2_cyclicPG(2,q)})--(\ref{eq2_cyclicRuzsa}) below are
obtained from $(v,k)$ modular Golomb rulers \cite
[Ch.\thinspace 5]{Dimit},\cite{Draka},\cite[Sec.\thinspace
5]{Gropp-nk},\cite [Sec.\thinspace 19.3]{Shearer-Handb}, see
Theorem \ref{Th1_2L+1}(ii).
\begin{eqnarray}
\text{cyclic }v_{k} &:&v=q^{2}+q+1,\hspace{0.2cm}k=q+1-\delta ,\hspace{0.2cm}
q+1>\delta \geq 0;  \label{eq2_cyclicPG(2,q)} \\
\text{cyclic }v_{k} &:&v=q^{2}-1,\hspace{0.2cm}k=q-\delta ,\hspace{0.2cm}
q>\delta \geq 0;  \label{eq2_cyclicAG(2,q)} \\
\text{cyclic }v_{k} &:&v=p^{2}-p,\hspace{0.2cm}k=p-1-\delta ,\hspace{0.2cm}
p-1>\delta \geq 0.  \label{eq2_cyclicRuzsa}
\end{eqnarray}
The configurations giving rise to (\ref{eq2_cyclicPG(2,q)}) use
the incidence matrix of the cyclic projective plane $PG(2,q)$
\cite{Singer},\cite[Sec.\thinspace 5.5]
{Dimit},\cite{Draka},\cite[Th.\thinspace 19.15]{Shearer-Handb}.
The family with parameters (\ref{eq2_cyclicAG(2,q)}) is
obtained from the \emph{cyclic starred affine plane} $AG(2,q)$
\cite{Bose},\cite[Sec.\thinspace
5.6]{Dimit},\cite{Draka,Funk2008},\cite[Th.\thinspace
19.17]{Shearer-Handb}, see also \cite[Ex.\thinspace
5]{DGMP-GraphCodes} and \cite{FunkLabNap} where the
configurations are called \emph{anti-flags}. We recall that the
starred plane $AG(2,q)$ is the affine plane without the origin
and the lines through the origin. Finally, the configurations
with parameters (\ref{eq2_cyclicRuzsa}) follow from Ruzsa's
construction \cite{Ruzsa},\cite[Sec.\thinspace
5.4]{Dimit},\cite{Draka},\cite[Th.\thinspace
19.19]{Shearer-Handb}.

The non-cyclic families with parameters (\ref{eq2_q-cancel})
and (\ref{eq2_q-1-cancel}) are given in \cite[Constructions
(i),(ii), p.\thinspace 126]{AFLN-graphs} and
\cite[Constructions 3.2,3.3, Rem.\thinspace 3.5]{GH}, see also
the references therein and~\cite{AfDaZ},\cite[Sec.\thinspace
3]{AfDaZ-InfProc},\cite[Sec.\thinspace 7.3]{DGMP-GraphCodes}.
\begin{eqnarray}
v_{k} &:&v=q^{2}-qs,\hspace{0.2cm}k=q-s-\delta ,\hspace{0.2cm}q>s\geq 0,
\hspace{0.2cm}q-s>\delta \geq 0;  \label{eq2_q-cancel} \\
v_{k} &:&v=q^{2}-(q-1)s-1,\hspace{0.2cm}k=q-s-\delta ,\hspace{0.2cm}q>s\geq
0,\hspace{0.2cm}q-s>\delta \geq 0.  \label{eq2_q-1-cancel}
\end{eqnarray}
In the projective plane $PG(2,q)$ we fix a line $\ell $ and a
point $P$ and assign an integer $s\geq 0$. If $P\in \ell $ we
choose $s$ points on $\ell $ distinct from $P$, and $s$ lines
through $P$ distinct from $\ell $. If $P\notin \ell $ we choose
$s$ arbitrary points on $\ell $ and consider the $s$ lines
connecting $P$ with these points. The incidence structure
obtained from $PG(2,q)$ by dismissing all the lines through the
$s+1$ selected points and all the points lying on the $s+1$
selected lines provides the family of (\ref{eq2_q-cancel}) if $
P\in \ell $ and the family with parameters
(\ref{eq2_q-1-cancel}) if $P\notin \ell $. For $s=0,$ the
construction of (\ref{eq2_q-1-cancel}) is given in
\cite{MePaWolk}. In~\cite{AfDaZ},\cite[Eqs
(3.2),(3.3)]{AfDaZ-InfProc}, the family with parameters (\ref
{eq2_q-cancel}) is described by using a block structure of the
incidence matrix of the affine plane $AG(2,q),$ see the
Extension Construction below. Configurations $(q^{2})_{q}$ and
$(q^{2}-1)_{q}$ are mentioned in many papers, see e.g.
\cite{FunkLabNap},\cite[Sec.\thinspace 5]{Gropp-nk}.

For $q$ a square, in \cite[Conjec.\thinspace 4.4,
Rem.\thinspace 4.5, Ex.\thinspace 4.6] {AFLN-graphs} and
\cite[Construction 3.7, Th.\thinspace 3.8]{GH},  families of
non-cyclic configuration $v_{k}$ with parameters
(\ref{eq2_Baer}) are provided; see also \cite[Ex.\thinspace
8]{DGMP-GraphCodes}. The configurations with parameters
(\ref{eq2_Baer2}) belong to these families; here, $c=q-
\sqrt{q}$. Configurations with parameters (\ref{eq2_Baer2}) are
also described  in \cite[Ex.\thinspace 2(ii)]{DGMP-ACCT2008}
and \cite{FunkLabNap}.
\begin{eqnarray}
v_{k} &:&v=c(q+\sqrt{q}+1),k=\sqrt{q}+c-\delta ,c=2,3,\ldots ,q-\sqrt{q}
,\delta \geq 0,q\text{ square};  \label{eq2_Baer} \\
v_{k} &:&v=q^{2}-\sqrt{q},\hspace{0.2cm}k=q-\delta ,\hspace{0.2cm}q>\delta
\geq 0,\text{ }q\text{ square}.  \label{eq2_Baer2}
\end{eqnarray}
In
\cite{AFLN-graphs,DGMP-ACCT2008,DGMP-GraphCodes,FunkLabNap,GH},
the partition $PG(2,q)$ into Baer subplanes for $q$ a square is
used; see also Example \ref{ex3.3_projPlane}(ii) of the present
work.

In \cite[Th.\thinspace 1.1]{FunkLabNap}, a family of non-cyclic
 with parameters
\begin{equation}
v_{k}:v=2p^{2},\hspace{0.2cm}k=p+s-\delta ,\hspace{0.2cm}0<s\leq q+1,\hspace{
0.2cm}\,q^{2}+q+1\leq p,\hspace{0.2cm}p+s>\delta \geq 0
\label{eq2_FuLabNabDecomp}
\end{equation}
is given. In \cite[Sec.\thinspace \thinspace
6]{DGMP-GraphCodes}, based on the cyclic starred affine plane,
a construction of non-cyclic configuration with parameters
\begin{eqnarray}
v_{k} &:&v=c(q-1),\hspace{0.2cm}k=c-\delta ,\hspace{0.2cm}c=2,3,\ldots ,b,
\hspace{0.2cm}b=q\,\,\mathrm{if}\,\,\delta \geq 1,\,\,
\label{eq2_affine_q-1} \\
b &=&\left\lceil \frac{q}{2}\right\rceil \,\,\mathrm{if}\,\,\delta =0,\text{
}c>\delta \geq 0,  \notag
\end{eqnarray}
is provided.

In \cite[Sec.\thinspace 2]{DGMP-ACCT2008},\cite[Sec.\thinspace
\thinspace 3]{DGMP-GraphCodes}, the following geometrical
construction which uses  point orbits under the action of a
collineation group is described.\smallskip

\noindent\textbf{Construction A}. Take any point orbit
$\mathcal{P}$ under the action of a collineation group in an
affine or projective space of order $q$. Choose an integer
$k\leq q+1$ such that the set $\mathcal{L}(\mathcal{P},k)$ of
lines meeting $\mathcal{P}$ in precisely $k$ points is not
empty. Define the following incidence structure: the points are
the points of $ \mathcal{P}$, the lines are the lines of
$\mathcal{L}(\mathcal{P},k)$, the incidence is that of the
ambient space.

\begin{theorem}
\label{Th3_constrOneOrb} In Construction A the number of lines
of $\mathcal{L }(\mathcal{P},k)$ through a point of
$\mathcal{P}$ is a constant $r_{k}$. The incidence structure is
a configuration $(v_{r_{k}},b_{k})$ with $v=|\mathcal{P}|$,
$b=|\mathcal{L}(\mathcal{P},k)|$.
\end{theorem}

By Definition \ref{Def1-Configur}, if $r_{k}=k$ Construction A
produces a symmetric configuration $v_{k}$.

It is noted in \cite{DGMP-ACCT2008,DGMP-GraphCodes} that
Construction A works for any $2$-$(v,k,1)$ design $D$ and for
any group of automorphisms of $D$. The role of $q+1$ is played
by the size of any block in$~D$.

Families of non-cyclic configuration $v_{k}$ obtained by
Construction A with the following parameters  are given in
\cite[Exs\thinspace 2,\thinspace 3] {DGMP-GraphCodes}.
\begin{eqnarray}
v_{k} &:&v=\frac{q(q-1)}{2},\hspace{0.2cm}k=\frac{q+1}{2}-\delta ,\hspace{
0.2cm}\frac{q+1}{2}>\delta \geq 0,\text{ }q\text{ odd}.
\label{eq2_ConstrA_GraphCode1} \\
v_{k} &:&v=\frac{q(q+1)}{2},\hspace{0.2cm}k=\frac{q-1}{2}-\delta ,\hspace{
0.2cm}\frac{q-1}{2}>\delta \geq 0,\text{ }q\text{ odd}.
\label{eq2_ConstrA_GraphCode2} \\
v_{k} &:&v=q^{2}+q-q\sqrt{q},\text{ }k=q-\sqrt{q},\text{ }q-\sqrt{q}>\delta
\geq 0,\text{ }q\text{ square.}  \label{eq2_ConstrA_GraphCode3}
\end{eqnarray}

In \cite{AfDaZ}, a construction method  for non-cyclic
configuration $v_{k}$ with parameters (\ref{eq2_tetaExten}) is
proposed, and called \textquotedblleft Construction\emph{\
}$\theta $-extension\textquotedblright. This construction is
also considered in \cite{AfDaZ-InfProc}, where it is called
CE-construction (\textquotedblleft
Cancellation+Enlargement\textquotedblright ). The terminology
we use here is \textquotedblleft \emph{Extension
Construction}\textquotedblright .
\begin{equation}
v_{k}:v=q^{2}-qs+\theta ,\text{ }k=q-s-\Delta ,\text{ }q>s\geq
0,\,q-s>\Delta \geq 0,\,\theta =0,1,\ldots ,q-s+1.  \label{eq2_tetaExten}
\end{equation}

We first describe the Extension Construction in geometrical
terms. Let $v_k$ be a configuration $(\mathcal P, \mathcal L)$
with incidence matrix $\mathbf{M}(v,k)$. Assume that there
exists  a set of $k-1$ pairwise disjoint lines of $v_k$, say
$\ell _{1},\ell _{2},\ldots ,\ell _{k-1}$, and a set of $k-1$
pairwise non-collinear points, say $P_{1},P_{2},\ldots
,P_{k-1}$, with the property that each $P_i$ belongs to
precisely one $\ell_{\pi(i)}$. Here $\pi$ denotes a permutation
of the indexes $1,2,\ldots,k-1$. We define a new incidence
structure $(\mathcal P',\mathcal L')$ as follows:
\begin{enumerate}
\item $\mathcal P'=\mathcal P\cup \{P_{\text{new}}\}$;
\item $\mathcal L'=\mathcal L\cup \{\ell_{\text{new}}\}$;
\item the lines incident with $P_{\text{new}}$ are $\ell_1,\ldots, \ell_{k-1}$ and $\ell_{\text{new}}$;
\item the points  incident with $\ell_{\text{new}}$ are
    $P_1,\ldots, P_{k-1}$ and $P_{\text{new}}$;
\item $P_i$ is not incident with $\ell_{\pi(i)}$;
\item for a point $P\in \mathcal P$ and a line $\ell \in
    \mathcal L$ we have that $P$ is incident with $\ell$ if
    and only if $P\in \ell$  in $v_k$, with the only $k-1$
    exceptions of $P=P_i$ and $\ell=\ell_{\pi(i)}$,
    $i=1,\ldots,k-1$.
\end{enumerate}
It is easy to check that $(\mathcal P',\mathcal L')$ is a configuration $(v+1)_k$.

It is interesting to note that this procedure can be viewed as
a generalization of a classical construction by V. Martinetti
for configurations $v_3$, going back to 1887 \cite{Martinetti}.
According to Martinetti's construction (quoted, e.g. in
\cite{Boben-v3,CarsDinStef-Reduc-n3,Gropp-Chemic},\cite[Sec.
2.4, Fig. 2.4.1] {Grunbaum}) two parallel lines $a$, $b$ and
two non collinear points $A$, $B$ are chosen so that $A\in a$,
$B\in b$. Then a line $c$ and a point $C$ are added. The points
$A$, $B$ are removed from the lines $a$ and $b$ and are
included into the new line $c$. The new point $C$ is included
into all lines $a$, $b$, and $c$.

Below we provide a description of the Extension Construction, as given in \cite{AfDaZ,AfDaZ-InfProc}.

\begin{definition}
\label{def2_ExtenAgeg}\cite{AfDaZ,AfDaZ-InfProc}\emph{\ }Let
$\mathbf{ M}(v,k)$ be an incidence matrix of a symmetric
configuration$~v_{k}.$ In  $\mathbf{M}(v,k)$, we consider an
aggregate $\mathcal{A}$ of $k-1$ rows corresponding to pairwise
disjoint lines of $v_k$  and $k-1$ columns  corresponding to
pairwise non-collinear points of $v_k$. The $(k-1)\times (k-1)$
submatrix $\mathbf{C}{(\mathcal{A})}$ formed by the
intersection of the rows and columns of $\mathcal{A}$ is called
a \emph{critical submatrix} of$~\mathcal{A}$. The aggregate
$\mathcal{A}$ is called an \emph{extending aggregate} (or
\emph{E-aggregate}) if its critical submatrix
$\mathbf{C}{(\mathcal{A})}$ is a permutation matrix $\mathbf{P}
_{k-1}$. The matrix $\mathbf{M}(v,k)$ \emph{admits an
extension} if it contains at least one E-aggregate. The matrix
$\mathbf{M}(v,k)$ \emph{admits }$\theta $\emph{\ extensions} if
it contains $\theta $ E-aggregates that do not intersect each
other. We also will say that a configuration$~v_{k}$
\emph{admits an extension} or \emph{admits }$\theta $\emph{\
extensions} if its incidence matrix does.
\end{definition}

\textbf{Procedure E }(\emph{Extension Procedure}). Let
$\mathbf{M} (v,k)=[m_{ij}]$ be an incidence matrix of a
symmetric configuration $v_{k}=(\mathcal P,\mathcal L).$ Assume
that $\mathbf{M}(v,k)$ admits an extension.

\begin{enumerate}
\item To the matrix $\mathbf{M}(v,k)$,  add a new row from
    below and a new column to the right. Denote the new
    $(v+1)\times (v+1)$ matrix by $\mathbf{B}=[b_{ij}]$,
    and let $b_{v+1,v+1}=1$ whereas $b_{v+1,1}=\ldots
    =b_{v+1,v}=0,$ $b_{1,v+1}=\ldots =b_{v,v+1}=0.$

\item One of E-aggregates of $\mathbf{M}(v,k)$, say
    $\mathcal{A},$ is chosen. In the matrix $\mathbf{B,}$
    we \textquotedblleft clone\textquotedblright\ all $k-1$
    units of the critical submatrix $\mathbf{C
    }{(\mathcal{A})}$ writing their \textquotedblleft
    projections\textquotedblright\ to the new row and
    column. Finally, the units cloned are changed by
    zeroes. In other words, let  the aggregate
    $\mathcal{A}$ consist of rows with indexes $i_{u},$
    $u=1,2,\ldots ,k-1,$ and columns with indexes $j_{d},$
$d=1,2,\ldots ,k-1.$ Then the units of $
\mathbf{C}{(\mathcal{A})}$ are as follows:
 $m_{i_{u}j_{\pi(u)}}=1,$ $u=1,2,\ldots ,k-1,$ for some
permutation $\pi$ of the indexes $1,\ldots,k-1$.  Then
$\mathbf B$ arising from Step 1 is changed as follows:
$b_{i_{u},v+1}=1,$ $b_{v+1,j_{d}}=1,$
$b_{i_{u}j_{\pi(u)}}=0,$ $u=1,2,\ldots ,k-1$, $d=1,2,\ldots
,k-1$.
\end{enumerate}
It is easily seen that $\mathbf B$ is an incidence matrix for $(\mathcal P', \mathcal L')$. Therefore, the following result can be easily proved.
\begin{theorem}
\label{th2_thetaExten}\emph{\cite{AfDaZ,AfDaZ-InfProc} }Let
$\mathbf{M }(v,k)$ be an incidence matrix of a symmetric
configuration$~v_{k}.$ Assume that $\mathbf{M}(v,k)$ admits
$\theta $ extensions, for some $\theta \geq 1$.

\begin{description}
\item[(i)] $\theta $ repeated applications of Procedure E to $\mathbf{M}(v,k)$
gives an incidence matrix $\mathbf{M}(v+\theta ,k)$ of a symmetric
configuration $(v+\theta )_{k}.$

\item[(ii)] If $\theta \geq k-1$, then any $k-1$ new rows and $k-1$ new
columns obtained as a result of repeated application of Procedure E form an
E-aggregate.
\end{description}
\end{theorem}

In \cite{AfDaZ,AfDaZ-InfProc} the  Extension Construction is
applied to affine planes and provides configuration with
parameters as in \eqref{eq2_tetaExten}. Let $(x_1,x_2)$ denote
coordinates for the affine plane $AG(2,q)$. The incidence
structure of $AG(2,q)$ is a resolvable 2-$(q^{2},q,1)$-design
with $q+1$ resolution classes. Each class contains $q$ parallel
lines. The $q$ classes are lines with equation $x_{2}=wx_{1}+u$
where $w\in F_{q}$ is a constant for the given class and $u$
runs over $ F_{q} $. One more class contains $q$ lines
$x_{1}=c$. This class is removed from $AG(2,q)$ in order to
obtain a symmetric configuration $(q^2)_q$, whose incidence
matrix $\mathbf M(q^2,q)$ can be represented as a superposition
of $q^{2}$ permutation matrices $\mathbf{P}_{q}$. Each block
row contains one resolution class. Each block column
corresponds to $q$ points $(d,x_{2})$ where $d$ is a constant
for the given block column and $x_{2}$ runs over $F_{q}$. Then
from $\mathbf{M(}q^{2},q\mathbf{)}$ one removes $s$ block rows
and columns. An incidence matrix
$\mathbf{M(}q^{2}-qs,q-s\mathbf{)}$ is obtained. It is a
superposition of $(q-s\mathbf{)}^{2}$ matrices $\mathbf{P}
_{q}$. Further, a $(q-s)\times (q-s)$ 01-matrix
$\mathbf{S}_{\Delta }$ with $ \Delta $ units in every row and
column is taken. In $\mathbf{M(} q^{2}-qs,q-s) $, submatrices
$\mathbf{P}_{q}$ marked by units of $\mathbf{S} _{\Delta }$ are
changed by $\mathbf{0}_{q}$. An incidence matrix $
\mathbf{M(}q^{2}-qs,q-s-\Delta ).$ is obtained; it admits
$\theta \leq q-s$ extensions. When Procedure E is executed by
$q-s$ times,  Procedure E can be applied once more, according
to Theorem \ref{th2_thetaExten}(ii).

Some known results on existence and non-existence of sporadic
symmetric configurations will be mentioned in Sections
\ref{sec_ParamCyclicConfig} and \ref{sec6_spectrum}.

We end this section by remarking that cyclic symmetric
configurations can be constructing from Sidon sets. Sidon sets
are combinatorial objects equivalent  to Golomb rulers.
\begin{definition}
\label{def1_Sidon}\cite{Dimit,BibliogrSidon} A \emph{Sidon }$k$-\emph{set}
 (respectively, $(v,k)$ \emph{modular Sidon set}) is an ordered set of $
k $ integers $(a_{1},a_{2},\ldots ,a_{k})$ such that $0\leq
a_{1}<a_{2}<\ldots <a_{k}$ and all pairwise sums $\{a_{i}+a_{j}\,|\,1\leq
i\leq j\leq k\}$ are different (respectively, different modulo $v$).
\end{definition}

Sidon sets are called also \emph{Sidon sequences}, or
$B_{2}$\emph{ sequence}; see \cite{Dimit,BibliogrSidon} and the
references therein for more details and terminology. It should
be noted that in Sidon sets we consider sums
$a_{i}+a_{j}$\thinspace of not necessarily distinct elements.

The relation between Sidon sets and Golomb rulers is described
in the following well-known result (for a proof see e.g.
\cite[Ch.\thinspace 4]{Dimit}).

\begin{theorem}
\label{th1_Sidon=Golomb}A Sidon $k$-set (respectively, $(v,k)$ modular Sidon
set) is a Golomb ruler of order $k$ (respectively, $(v,k)$ modular Golomb
ruler), and conversely.
\end{theorem}

The smallest possible value of $v$ for which a $(v,k)$ modular
Sidon set exists coincides with $v_{\delta }(k)$. This makes
our notation consistent with \cite{GrahSloan}. General bounds
on $v_{\delta }(k)$ and precise results for smaller $k$'s can
be found in
\cite{GrahSloan,OstergModulSidon,ShearerWebModulGR,Swanson}.

\section{\label{sec_doub-circ}Constructions of block double-circulant
incidence matrices $\mathbf{M}(v,k)$}

The aim of this section is the construction of BDC incidence
matrices of cyclic symmetric configurations. A method based on
the Golomb ruler associated to a cyclic simmetric configuration
is described in Subsection \ref{subsec_permutInt}: splitting a
starting modular Golomb ruler  a number of quotient Golomb
rulers forming a needed BDC matrix are obtained. The same
method can be described in terms of the action of the
automorphism group of the configuration, see Subsection
\ref{Subsec3.3_CyclicAutomorph}. We provide two different
descriptions because when dealing with a given configuration
usually either one approach or the other can be more
conveniently used. For example, ideas of Subsection
\ref{subsec_permutInt} work better if the  Golomb ruler
associated to a configuration is described explicitly as a list
of integers. The approach of Subsection
\ref{Subsec3.3_CyclicAutomorph} is useful for instance when the
configuration arises  from geometrical objects such as cyclic
projective and affine planes. Sometimes both the approaches can
be conveniently used, cf. Examples \ref{ex_3.2_AffPlaneLing}
and~\ref{ex3.3_AfPlane}(i).

Throughout this section,
\begin{equation}
(a_{1},a_{2},\ldots ,a_{k})\text{ is a }(v,k)\text{ modular Golomb ruler},%
\text{ with }v=td \text{ for integers }t,d.  \label{eq3_t=v/s}
\end{equation}

\subsection{\label{subsec_bdc&families}BDC matrices $\mathbf{M}(v,k)$ and
families of symmetric configurations}

\begin{definition}
\label{def3.1_block double-circul}Let $v=td.$ A $v\times v$
matrix $\mathbf{A }$ is said to be a \emph{block
double-circulant} \emph{matrix }(or \emph{BDC matrix}) if
\begin{equation}
\mathbf{A}=\left[ \renewcommand{\arraystretch}{0.85}
\begin{array}{cccc}
\mathbf{C}_{0,0} & \mathbf{C}_{0,1} & \ldots & \mathbf{C}_{0,t-1} \\
\mathbf{C}_{1,0} & \mathbf{C}_{1,1} & \ldots & \mathbf{C}_{1,t-1} \\
\vdots & \vdots & \vdots & \vdots \\
\mathbf{C}_{t-1,0} & \mathbf{C}_{t-1,1} & \ldots & \mathbf{C}_{t-1,t-1}
\end{array}
\right] ,\text{ }  \label{eq3.1 _block-circulant}
\end{equation}
where $\mathbf{C}_{i,j}$ is a \emph{circulant }$d\times d$
0,1-matrix for all $i,j$, and  submatrices $\mathbf{C}_{i,j}$
and $\mathbf{C}_{l,m}$ with $ j-i\equiv m-l\pmod t$ have equal
weights. The matrix
\begin{equation}
\mathbf{W(A)}=\left[ \renewcommand{\arraystretch}{0.80}
\begin{array}{ccccccc}
w_{0} & w_{1} & w_{2} & w_{3} & \ldots & w_{t-2} & w_{t-1} \\
w_{t-1} & w_{0} & w_{1} & w_{2} & \ldots & w_{t-3} & w_{t-2} \\
w_{t-2} & w_{t-1} & w_{0} & w_{1} & \ldots & w_{t-4} & w_{t-3} \\
\vdots & \vdots & \vdots & \vdots & \vdots & \vdots & \vdots \\
w_{1} & w_{2} & w_{3} & w_{4} & \ldots & w_{t-1} & w_{0}
\end{array}
\right]  \label{eq3.1_block-circulant-weights}
\end{equation}
is a \emph{circulant }$t\times t$ matrix whose entry in
position $i,j$ is the \emph{weight} of $\mathbf{C}_{i,j}$.
$\mathbf{W(A)}$ is called the \emph{weight matrix} of
$\mathbf{A.}$ The vector $\overline{\mathbf{W}}\mathbf{(A)}=(
w_{0},w_{1},\ldots ,w_{t-1}\mathbf{)}$ is called the
\emph{weight vector} of $ \mathbf{A.}$
\end{definition}

We present some simple techniques for  obtaining BDC matrices
of symmetric configurations from a given BDC $v\times v$ matrix
$\mathbf{A}$ of (\ref{eq3.1 _block-circulant}) with weight
matrix $\mathbf{W(A)}$ of (\ref
{eq3.1_block-circulant-weights}). We assume that $v=td.$

\begin{description}
\item[(i)] For $h\in \{0,1,\ldots ,t-1\}$, dismiss
 $\delta _{h}\geq 0$ units in each row of every submatrix
$\mathbf{C}_{i,j}$ with $j-i\equiv h\pmod t$, in such a way
that the the obtained submatrix is still circulant. A BDC
matrix $\mathbf{A}^{\prime }$ is then obtained; it consists
of circulant matrices $ \mathbf{C}_{i,j}^{\prime }$ of
weight $w_{h}^{\prime }=w_{h}-\delta _{h},$ where
$j-i\equiv h \pmod t$ and $ h=0,1,\ldots ,t-1.$ It is an
incidence BDC matrix of a configuration $ v_{k^{\prime
}}^{\prime }$ with
$\overline{\mathbf{W}}(\mathbf{A}^{\prime })=(w_{0}^{\prime
},w_{1}^{\prime },\ldots ,w_{t-1}^{\prime }),$
\begin{equation}
v^{\prime }=v,\text{ }k^{\prime }=k-\sum_{h=0}^{t-1}\delta _{h},\text{ }
0\leq \delta _{h}\leq w_{h} ,\,w_{h}^{\prime }=w_{h}-\delta _{h}.
\label{eq3.1_rem(i)}
\end{equation}

\item[(ii)] Fix some non-negative integer $j\le t-1$. Let
    $m$ be such that $w_{m}\leq w_{h}$ for all $h\neq j$.
    Cyclically shift all block rows of $\mathbf{A}$ to the
    left by $j$ block positions. A matrix $\mathbf{A}^{\ast
}$ with $ \overline{\mathbf{W}}(\mathbf{A}^{\ast
})=(w_{0}^{\ast }=w_{j},w_{1}^{\ast }=w_{j+1},\ldots
,w_{u}^{\ast }=w_{u+j\pmod t},\ldots ,w_{t-1}^{\ast
}=w_{j-1})$ is obtained. By applying  (i), construct a
matrix $\mathbf{A}^{\ast \ast }$ with $w_{0}^{\ast \ast
}=w_{0}^{\ast }=w_{j},$ $w_{h}^{\ast \ast }=w_{m},$ $h\geq
1.$ Now remove from $\mathbf{A} ^{\ast \ast }$ $t-c$ block
rows and columns from the bottom and the right. In this way
a $cd\times cd$ BDC matrix $\mathbf{A}^{\prime }$ is
obtained,  with $\overline{ \mathbf{W}}(\mathbf{A}^{\prime
})=(w_{j},w_{m},\ldots ,w_{m})$. It is an incidence matrix
of a configuration $v_{k^{\prime }}^{\prime }$ with
\begin{equation}
v^{\prime }=cd,\,k^{\prime }=w_{j}+(c-1)w_{m},\,c=1,2,\ldots ,t.
\label{eq3.1_rem(ii)}
\end{equation}

\item[(iii)] Let $t$ be even. Let $\mathbf{A}^{\ast }$ be
    as in (ii). Let $w_{ \text{od}},w_{\text{ev}}$ be
    weights such that $w_{\text{od}}\leq w_{h}^{\ast }$ for
    odd $h=1,3,\ldots ,t-1,$ and  $w_{\text{ev}}\leq
    w_{h}^{\ast }$ for even $h=2,4,\ldots ,t-2$. By
    applying  (i), construct a matrix $\mathbf{A}^{\ast
    \ast }$ with  $w_{0}^{\ast \ast }=w_{0}^{\ast }=w_{j},$
    $w_{h}^{\ast \ast }=w_{\text{od}}$ for odd $h$,
$w_{h}^{\ast \ast }=w_{\text{ev}}$ for even $ h\geq 2$. {}
From $\mathbf{A}^{\ast \ast }$ remove $t-2f$ block rows and
columns from the bottom and the right. A $2fd\times 2fd$
BDC matrix $\mathbf{A}^{\prime }$ with
$\overline{\mathbf{W}}(\mathbf{A}^{\prime
})=(w_{j},w_{\text{od}},\underbrace{w_{\text{ev}},w_{\text{od}},\ldots
,w_{ \text{ev}},w_{\text{od}}}_{f-1\text{ pairs}})$ is
obtained. It is an incidence matrix of a configuration
$v_{k^{\prime }}^{\prime }$ with
\begin{equation}
v^{\prime }=2fd,\text{ }k^{\prime }=w_{j}+w_{\text{od}}+(f-1)(w_{\text{ev}
}+w_{\text{od}}),\text{ }f=1,2,\ldots ,t/2.  \label{eq3.1_rem(iii)}
\end{equation}
\end{description}

Other methods for obtaining families of symmetric
configurations from $\mathbf{A}$ of~(\ref{eq3.1
_block-circulant}) can be found in \cite[Sec.\thinspace
4]{DGMP-GraphCodes}.

\subsection{\label{subsec_permutInt}Using permutations of the set of
integers $\{0,1,\ldots ,v-1\}$}

In this subsection we show a method to obtain  BDC matrices
from any $(v,k)$ modular Golomb ruler with $v$ a composite
integer (see Theorem \ref{th3.2_Asigma-s} below). A key tool is
the notion of quotient modular Golomb ruler, as introduced in
\cite{LingDTSAP} and \cite[p.\thinspace 3]{ShearerReport}; it
should be noted that quotient rulers are used in
\cite{LingDTSAP,ShearerReport} with a different goal, that is,
in order to obtain difference triangle sets. We will construct
a permutation $\sigma$ of the set of indexes of points (and
lines) of the cyclic configuration $v_k$ associated to the
original modular Golomb ruler, such that the incidence matrix
$\mathbf A_{\sigma}$ of $v_k$ corresponding to  $\sigma$ (cf.
Definition \ref{def3.2_Asigma}) is a BDC matrix whose blocks
correspond to the quotients of the original ruler.

We now sketch the construction of quotient rulers, as given in \cite{LingDTSAP,ShearerReport}.
For the ruler (\ref{eq3_t=v/s}), and for any  $h=0,1,\ldots ,t-1$, let
\begin{equation}
B_{h}=\bigg\{\frac{a_{i}-h}{t}\mid a_{i}\equiv h\pmod t\bigg\},\text{ }
w_{h}=|B_{h}|.
\label{eq3.2_Ah_Bh}
\end{equation}
Clearly, $\sum_{h=0}^{t-1}w_{h}=k$ holds.
\begin{theorem}
\label{th3.2_Bh=GR}\emph{\cite{LingDTSAP,ShearerReport} }For every $
h=0,\ldots ,t-1$, $B_{h}$ of \emph{(\ref{eq3.2_Ah_Bh}) }is a $(d,w_{h})$
modular Golomb ruler.
\end{theorem}

\begin{definition}
\label{def3.2_Asigma}Let $(a_{1},a_{2},\ldots ,a_{k})$ be a $(v,k)$ modular
Golomb ruler. For each $u=0,1,\ldots ,v-1,$ let
\begin{equation}
L_{u}=\{a_{1}+u\pmod v,\text{ }a_{2}+u\pmod v,\ldots ,a_{k}+u\pmod v\}.\label{eq3.2_Bu}
\end{equation}
For a permutation $\sigma$ of the set $\{0,1,\ldots ,v-1\}$,
a $v\times v$ 01-matrix $\mathbf{A}_{\sigma }$ is defined as follows. Let $
i,j\in \{0,1,\ldots ,v-1\}$. The element in position $(i,j)$ of $\mathbf{A}
_{\sigma }$ is $1$ if and only if $\sigma (j)\in L_{\sigma (i)}$ (or,
equivalently, if and only if $\sigma (j)-\sigma (i)\pmod v\in L_{0}$).
\end{definition}

\begin{lemma}
\label{lem3.2_Asigma=M(v,k)-J2free}For every choice of $\sigma $, the matrix
$\mathbf{A}_{\sigma }$ of Definition \emph{\ref{def3.2_Asigma} }is a $
\mathbf{J}_{2}$-free incidence matrix $\mathbf{M}(v,k)$ of a symmetric
configuration $v_{k}$.
\end{lemma}

\begin{proof}
By Theorem \ref{Th1_2L+1}(ii), the matrix $\mathbf A_{id}$ is
the incidence matrix of a cyclic symmetric configuration $v_k$.
It is easily seen that $\mathbf A_{\sigma}$ is a different
incidence matrix of the same $v_k$ (points and lines are
rearranged according to $\sigma$).
\end{proof}

\begin{theorem}
\label{th3.2_Asigma-s}Let $(a_{1},a_{2},\ldots ,a_{k})$ be a
$(v,k)$ modular Golomb ruler with $v=td.$ Let $\sigma _{t}$ be
the permutation of the set $ \{0,1,\ldots ,v-1\}$ such that
\begin{equation}
\sigma _{t}(ad+b)=bt+a\text{ for }0\leq a\leq t-1,\text{ }0\leq b\leq d-1.
\label{eq3.2_sigma_s}
\end{equation}
Let $B_{h}$ and $w_{h}$ be as in \emph{(\ref{eq3.2_Ah_Bh})},
and $\mathbf{A}_{\sigma _{t}}$ be as in Definition
\emph{\ref{def3.2_Asigma}}. Also, let
$\mathbf{M}_{0}\mathbf{,M}_{1}\mathbf{ ,\ldots ,M}_{t-1}$ be
the $d\times d$ blocks of $\mathbf{A}_{\sigma _{t}}$ such that
the first $d$ rows of $\mathbf{A}_{\sigma _{t}}$ are a block
row $ \left[ \mathbf{M}_{0}\mathbf{M}_{1}\mathbf{\ldots
M}_{t-1}\right] .$
Finally, let $\mathbf{T}_{1}\mathbf{,T}_{2}\mathbf{,\ldots ,T}_{t-1}$  be
 the $d\times d$ blocks of $\mathbf{A}_{\sigma _{t}}$ such that the
first $d$ columns of $\mathbf{A}_{\sigma _{t}}$ are a block column $\left[
\mathbf{M}_{0}\mathbf{T}_{t-1}\mathbf{\ldots T}_{2}\mathbf{T}_{1}\right]
^{tr}$. Then

\begin{description}
\item[(i)] Each matrix\textbf{\ }$\mathbf{M}_{h}$ is a \emph{circulant} $
d\times d$ $01$-matrix of weight $w_{h}$. The first row of $\mathbf{M}_{h}$
corresponds to the $(d,w_{h})$ modular Golomb ruler $B_{h}$.

\item[(ii)] Each matrix $\mathbf{T}_{h}$ is a \emph{circulant} $d\times d$ $
01$-matrix of weight $w_{h}$ obtained from $\mathbf{M}_{h}$ by a cyclic
shift of rows to the right by one position.

\item[(iii)] The matrix $\mathbf{A}_{\sigma _{t}}$ is a block
double-circulant incidence matrix $\mathbf{M}(v,k)$ of a symmetric
configuration $v_{k}$ with the following structure:
\begin{equation}
\mathbf{A}_{\sigma _{t}}=\left[ \renewcommand{\arraystretch}{1.0}
\begin{array}{cccccc}
\mathbf{M}_{0} & \mathbf{M}_{1} & \mathbf{M}_{2} & \mathbf{\ldots } &
\mathbf{M}_{t-2} & \mathbf{M}_{t-1} \\
\mathbf{T}_{t-1} & \mathbf{M}_{0} & \mathbf{M}_{1} & \mathbf{\ldots } &
\mathbf{M}_{t-3} & \mathbf{M}_{t-2} \\
\mathbf{T}_{t-2} & \mathbf{T}_{t-1} & \mathbf{M}_{0} & \mathbf{\ldots } &
\mathbf{M}_{t-4} & \mathbf{M}_{t-3} \\
\mathbf{\vdots } & \mathbf{\vdots } & \mathbf{\vdots } & \mathbf{\vdots } &
\mathbf{\vdots } & \mathbf{\vdots } \\
\mathbf{T}_{2} & \mathbf{T}_{3} & \mathbf{T}_{4} & \mathbf{\ldots } &
\mathbf{M}_{0} & \mathbf{M}_{1} \\
\mathbf{T}_{1} & \mathbf{T}_{2} & \mathbf{T}3 & \mathbf{\ldots } & \mathbf{T}
_{t-1} & \mathbf{M}_{0}
\end{array}
\right]
. \label{eq3.2_incid-matrAsigma-s}
\end{equation}
\end{description}
\end{theorem}

\begin{proof}

\begin{description}
\item[(i)] Let $i,j\in \{0,1,\ldots ,d-1\},$ $h=0,1,\ldots ,t-1$. The
position $(i,j)$ in $\mathbf{M}_{h}$ is the position $(i,hd+j)$ in
$\mathbf{A}_{\sigma _{t}}$. The value $1$ appears in this position if and
only if $\sigma _{t}(hd+j)\in L_{\sigma _{t}(i)}$. By (\ref{eq3.2_sigma_s})
and (\ref{eq3.2_Bu}), $\sigma _{t}(hd+j)=jt+h,$ $\sigma _{t}(i)=it,$ and $
L_{\sigma _{t}(i)}=\{a_{1}+it\pmod v,\ldots ,a_{k}+it\pmod v\}.$ So, $\sigma
_{t}(hd+j)\in L_{\sigma _{t}(i)}$ if and only~if
\begin{equation}
jt\in \{a_{1}-h+it\pmod v,\ldots ,a_{k}-h+it\pmod v\}.
\label{eq3.2_condit-js}
\end{equation}
If $a_{u}\not\equiv h\pmod t$ then $t\nmid (a_{u}-h+it)$ and  $
jt=a_{u}-h+it\pmod v$ cannot occur.
Therefore, the
condition (\ref{eq3.2_condit-js}) is equivalent to
\begin{equation}
j\in \bigg\{\frac{a_{u}-h}{t}+i\pmod d|a_{u}\equiv h\pmod t\bigg\},  \label{eq3.2_Mh}
\end{equation}
which proves the assertion.

\item[(ii)] Let $i,j\in \{0,\ldots ,d-1\},$ $h=1,2,\ldots ,t-1$. The
position $(i,j)$ in $T_{h}$ is the position $((t-h)d+i,j)$ in $A_{\sigma
_{t}}$. The value $1$ appears in this position if and only if $\sigma
_{t}(j) $ belongs to $L_{\sigma _{t}((t-h)d+i)}$. Since $\sigma
_{t}((t-h)d+i)=it+t-h $, we have
\begin{equation*}
L_{\sigma _{t}((t-h)d+i)}=L_{it+t-h}=\{a_{1}+it+t-h\pmod v,\ldots
,a_{k}+it+t-h\pmod v\}.
\end{equation*}
Then $\sigma _{t}(j)=jt$ belongs to $L_{\sigma _{t}((t-h)d+i)}$ if and only
if
\begin{equation*}
jt\in \{a_{1}+it+t-h\pmod v,\ldots ,a_{k}+it+t-h\pmod v\}.
\end{equation*}
Arguing as in (i), we obtain that this condition is equivalent to
\begin{equation}
j\in \bigg\{\frac{a_{u}-h}{t}+i+1\pmod d|a_{u}\equiv h\pmod t\bigg\},  \label{eq3.2_Th}
\end{equation}
which proves the assertion.

\item[(iii)] We need to show that  for every
pair $(i,j)$, $i,j=0,1,\ldots ,v-d-1$, the value in position $(i,j)$ in $\mathbf{A}_{\sigma _{t}}$ is
equal to that in position $(i+d,j+d)$. The value in position $(i,j)$ is
equal to $1$ if and only if $\sigma _{t}(j)\in L_{\sigma _{t}(i)}$. Write $
i=i_{1}d+i_{2}$, $j=j_{1}d+j_{2}$, with $0\leq i_{1},j_{1}\leq t-2$, $0\leq
i_{2},j_{2}\leq d-1$. Then $\sigma _{t}(j)=j_{2}t+j_{1}$ and $\sigma
_{t}(i)=i_{2}t+i_{1}$, and hence the value in position $(i,j)$ is $1$ if and
only if $j_{2}t+j_{1}-i_{2}t-i_{1}\pmod v\in L_{0}.$ Note that $
i+d=(i_{1}+1)d+i_{2}$ and $j+d=(j_{1}+1)d+j_{2}$. Then the value in position
$(i+d,j+d)$ is $1$ if and only if $j_{2}t+(j_{1}+1)-i_{2}t-(i_{1}+1)\pmod v
\in L_{0}.$ Since $(j_{1}+1)-(i_{1}+1)=j_{1}-i_{1}\pmod v$
holds, the assertion is proven.
\end{description}
\end{proof}

\begin{example}
\label{ex_3.2_Ruzsa}Let $p$ be a prime. Let $g$ be a primitive element of $
F_{p}$. The following Ruzsa's sequence \cite{Ruzsa},\cite[Sec.\thinspace 5.4]
{Dimit},\cite[Th.\thinspace 19.19]{Shearer-Handb} forms a $(p^{2}-p,p-1)$
modular Golomb ruler:
\begin{equation}
e_{u}=pu+(p-1)g^{u}\pmod{p^2-p},\text{ }u=1,2,\ldots
,p-1,\text{ }v=p^2-p.  \label{eq3.2_Ruzsa}
\end{equation}

\begin{description}
\item[(i)] In \cite[Tab.\thinspace 5]{ShearerReport}, a proper divisor of $
p-1$ is taken as $t$ to obtain new $(d,w_{h})$ modular Golomb rulers.
In this case, $d=p\frac{p-1}{t}$ and $%
w_{h}=\frac{p-1}{t}$ for every $h$ in (\ref{eq3.2_Ah_Bh}).
The matrix $\mathbf{A}_{\sigma _{t}}$ has a weight vector $
\overline{\mathbf{W}}(\mathbf{A}_{\sigma _{t}})=(\frac{p-1}{t},\ldots ,\frac{
p-1}{t}).$

\item[(ii)] BDC matrices such that each weight $w_h$ is in $\{0,1\}$ admit
an extension by Procedure E, see Section \ref{sec_known}, and hence
can be effectively used to obtain new families of configurations (cf.
Section \ref{sec_admitExten} and Example \ref{Ex_4_extension}(ii)). There are two different possibilities
to get a matrix $\mathbf{A}_{\sigma _{t}}$ with 01-weight vector from
\eqref{eq3.2_Ruzsa}.

a) Fix $t=p-1,$ $d=p$.
Then for each $h=0,1,\ldots ,t-1$ there is precisely one element $e_{u}$
such that $e_{u}\equiv h\pmod t.$ We have $e_{p-1}\equiv 0\pmod t$ and $
e_{u}\equiv u\pmod t,$ $u=1,2,\ldots ,p-2.$ Therefore,
$\overline{\mathbf{W}}(\mathbf{A}
_{\sigma _{p-1}})=(\underbrace{1,1,\ldots ,1}_{p-1}).$

b) Fix $t=p,$ $d=p-1$.
In this case $
e_{u}\not\equiv 0\pmod t$ for all $u.$ Also, for each $h=1,2,\ldots ,t-1$
there is precisely one element $e_{u}$ such that $e_{u}\equiv h\pmod t.$
We have $e_{u}\equiv h\pmod s$ if and only if $-g^{u}\equiv h\pmod p.$
Therefore,
$\overline{\mathbf{W}}(\mathbf{A}_{\sigma
_{p}})=(0,\underbrace{1,1,\ldots ,1}_{p-1}).$

\end{description}
\end{example}

\begin{example}
\label{ex_3.2_AffPlaneLing} Consider the $(q^{2}-1,q)$ modular Golomb ruler
obtained from the cyclic starred affine plane $AG(2,q)$ \cite{Bose}. Let $q^{2}-1=td.$
In \cite{LingDTSAP}, by using both counting arguments and properties of the
ruler as a difference set, it is proved that if $t$ is a divisor of $q+1$
then exactly $t-1$ values of $w_{h}$ are equal to $\frac{q+1}{t}$, and there exists precisely one $h_0$
with $w_{h_0}=\frac{q+1}{t}-1.$ In \cite
{LingDTSAP} only proper divisor $t$ of $q+1$ are considered, as this is the relevant case in connection with
 difference triangle sets. Yet, the same arguments
  work $t=q+1$, and hence on can obtain  a weight vector of $\mathbf{A}_{\sigma _{q+1}}$ consisting of zeroes
and units, and admitting an extension by Procedure E. Without loss of generality
$\overline{\mathbf{W}}(\mathbf{A}_{\sigma _{q+1}})=(0,\underbrace{1,1,\ldots
,1}_{q})$ can be assumed. For comparison, see also Example \ref{ex3.3_AfPlane} below.
\end{example}

\subsection{Using subgroups of the automorphism
group of a cyclic configuration \label{Subsec3.3_CyclicAutomorph}}

The geometrical interpretation of the procedure illustrated in
Subsection \ref{subsec_permutInt} was presented in
\cite{DGMP-ACCT2008,DGMP-GraphCodes}. Here, after summarizing
some of the results from \cite{DGMP-ACCT2008,DGMP-GraphCodes},
we   apply the procedure to cyclic configurations $(q^2-1)_q$
associated to affine planes $AG(2,q)$ for $t$ a divisor of
$q-1$, see Theorem \ref{th3.3_AfPln}.

For a \emph{cyclic }symmetric configuration $v_{k}$, viewed as an incidence
structure $\mathcal{I}=(\mathcal{P},\mathcal{L})$, let $\sigma$ be a permutation
 of $\mathcal{P}$ mapping lines to lines, and acting regularly on
both $\mathcal{P}$ and $\mathcal{L}$. Let $S$ be the cyclic
group generated by $\sigma $. Let $\mathcal{P}=\{P_{0},\ldots
,P_{v-1}\}$ and $\mathcal{L}=\{\ell _{0},\ldots ,\ell
_{v-1}\}$. Arrange indexes so that $\sigma :P_{i}\mapsto
P_{i+1\pmod v}$ and $\ell _{i}=\sigma ^{i}(\ell _{0}).$
Clearly, $P_{i}=\sigma ^{i}(P_{0})$ holds.

For any divisor $d$ of $v$, the group $S$ has a unique cyclic
subgroup $ \widehat{S}_{d}$ of order $d$, namely the group
generated by $\sigma ^{t}$ where $t=v/d$. Let $
O_{0},O_{1},\ldots ,O_{t-1}$ (resp. $L_{0},L_{1},\ldots
,L_{t-1}$) be the orbits of $\mathcal{P}$ (resp. $\mathcal{L}$)
under the action of $\widehat{S }_{d}$. Clearly,
$|O_{i}|=|L_{i}|=d$ for any $i.$ We arrange indexes so that
$P_{0}\in O_{0},$ $O_{w}=\sigma ^{w}(O_{0}),$ $\ell _{0}\in
L_{0},$ $ L_{w}=\sigma ^{w}(L_{0})$. For each $i=0,1,\ldots
,t-1$,
\begin{equation*}
O_{i}=\{P_{i},\sigma ^{t}(P_{i}),\sigma ^{2t}(P_{i}),\ldots ,\sigma
^{(d-1)t}(P_{i})\},\text{ }L_{i}=\{\ell _{i},\sigma ^{t}(\ell _{i}),\sigma
^{2t}(\ell _{i}),\ldots ,\sigma ^{(d-1)t}(\ell _{i})\}.
\end{equation*}
Equivalently, $O_{i}$ (resp. $L_{i}$) consists of $d$ points
$P_{u}$ (resp. $d$ lines $ L_{u}$) with $u$ equal to $i$ modulo
$t$.

Let
\begin{equation}
w_{u}=|\ell _{0}\cap O_{u}|,\text{ }u=0,1,\ldots ,t-1.  \label{eq3.3_wu}
\end{equation}
Clearly, $w_{0}+w_{2}+\ldots +w_{t-1}=k.$

\begin{theorem}
\label{Th3.3_the-same-meeting} \emph{\cite{DGMP-GraphCodes}}
Let $\mathcal{I} =(\mathcal{P},\mathcal{L})$ be a cyclic
symmetric configuration $v_{k}$ with $v=td$. Let $d$, $t$,
$\hat S_d$, $O_i$, $L_i$ be as above.

\begin{description}
\item[(i)] For any $i$ and $j,$ every line of the orbit
    $L_{i}$ meets the orbit $O_{j}$ in the same number of
points $w_{j-i\pmod t}$ where $w_{u}$ is defined by
\emph{(\ref{eq3.3_wu})}.

\item[(ii)] The incidence matrix of $\mathcal{I}$ is a
    block double-circulant matrix $\mathbf{A}$ of type
    \emph{(\ref{eq3.1 _block-circulant}) }where
    $\mathbf{C}_{i,j}$ is a \emph{circulant} $d\times d $
    matrix of weight $w_{j-i\pmod t}$, with $w_{u}$ as in
    \emph{(\ref{eq3.3_wu})}.
\end{description}
\end{theorem}

In order to use Theorem \ref{Th3.3_the-same-meeting}
effectively one should find intersection numbers of orbits of
the cyclic subgroup $\widehat{S}_{d}$. For cyclic projective
and starred affine planes useful results on these numbers are
given e.g. in \cite{DC2,DGMP-GraphCodes} and in the references
therein.

\begin{example}
\label{ex3.3_projPlane} We consider the projective plane $PG(2,q)$ as
a cyclic symmetric configuration
 $(q^{2}+q+1)_{q+1}$
 \cite{Singer},\cite[Sec.\thinspace 5]{DGMP-GraphCodes},\cite[Sec.\thinspace 5.5]{Dimit},\cite[Th.\thinspace 19.15]{Shearer-Handb}.
 In this case the group $S$ is a
Singer group of $PG(2,q)$.
\end{example}

\begin{description}
\item[(i)] Let $t=3$, $t|(q^{2}+q+1),$ $p\equiv 2\pmod 3,$
    and let $ \{i_{0},i_{1},i_{2}\}=\{0,1,2\}$. In
\cite[Prop.\thinspace 4] {DGMP-GraphCodes} the following is
proved: $w_{i_{0}}=(q+2\sqrt{q}+1)/3,$ $
w_{i_{1}}=w_{i_{2}}=(q-\sqrt{q}+1)/3,$ if $q=p^{4m+2}$;
$w_{i_{0}}=(q-2\sqrt{ q}+1)/3,$
$w_{i_{1}}=w_{i_{2}}=(q+\sqrt{q}+1)/3,$ if $q=p^{4m}$. Now
we use Theorem \ref{Th2-Mend} and (ii) of Subsection
\ref{subsec_bdc&families}. By (\ref{eq3.1_rem(ii)}) with
$c=2$, we obtain families of configurations $v_{k}$ with
parameters
\begin{eqnarray*}
v_{k} &:&v=2\frac{q^{2}+q+1}{3},\text{ }k=\frac{2q+\sqrt{q}+2}{3}-\delta ,
\text{ }\delta \geq 0,\text{ }q=p^{4m+2},\text{ }p\equiv 2\pmod 3; \\
v_{k} &:&v=2\frac{q^{2}+q+1}{3},\text{ }k=\frac{2q-\sqrt{q}+2}{3}-\delta ,
\text{ }\delta \geq 0,\text{ }q=p^{4m},\text{ }p\equiv 2\pmod 3.
\end{eqnarray*}

\item[(ii)] Let $q=p^{2m}$ be a square. Let $t$ be a prime
    divisor of $ q^{2}+q+1$. Then $t$ divides either
    $q+\sqrt{q}+1$ or $q-\sqrt{q}+1.$ Assume that $p\pmod
t$ is a generator of the multiplicative group of
${\mathbb{Z}} _{t}$. By \cite[Prop.\thinspace
6]{DGMP-GraphCodes}, in this case $w_{0}=(q+1\pm
(1-t)\sqrt{q})/t$, $w_{1}=w_{2}=\ldots =w_{t-1}=(q+1\pm
\sqrt{q})/t.$ Now we use Theorem \ref{Th2-Mend} and (ii) of
Subsection \ref{subsec_bdc&families}. By
(\ref{eq3.1_rem(ii)}), we obtain a family of configurations
$v_{k}$ with parameters
\begin{eqnarray}
v_{k} &:&v=c\frac{q^{2}+q+1}{t},\text{ }k=\frac{q+1\pm (1-t)\sqrt{q}}{t}
+(c-1)\frac{q+1\pm \sqrt{q}}{t}-\delta ,\text{ \quad }
\label{eq3.3_Ex(ii)ProjPlane} \\
c &=&1,2,\ldots ,t,\text{ }\delta \geq 0,\text{ }q=p^{2m},\text{ }t\text{
prime. }  \notag
\end{eqnarray}
The hypothesis that $p\pmod t$ is a generator of the
multiplicative group of ${\mathbb{Z}}_{t}$ holds e.g. in
the following cases: $q=3^{4},$ $t=7;$ $ q=2^{8},$ $t=13;$
$q=5^{4},$ $t=7;$ $q=2^{12},$ $t=19;$ $q=3^{8},$ $t=7;$ $
q=2^{16},$ $t=13;$ $q=17^{4},$ $t=7;$ $p\equiv 2\pmod t,$
$t=3.$

\item[(iii)] Let $q$ be a square. Let $v\geq 1,$
    $v|(q-\sqrt{q}+1),$ and $t= \frac{1}{v}(q-\sqrt{q}+1)$.
Then $d=v(q+\sqrt{q}+1)$ and,  by \cite[Prop.\thinspace
7]{DGMP-GraphCodes}, we have $w_{0}=\sqrt{q}+v,$ $
w_{1}=w_{2}=\ldots =w_{t-1}=v$. Now using (ii) of
Subsection \ref{subsec_bdc&families}, for $v=1$ we obtain a
family of configurations with parameters (\ref{eq2_Baer}).
The orbits $O_{0},O_{1},\ldots ,O_{t-1}$ are Baer
subplanes. Moreover, the case $ v=1$ admits an extension by
Procedure E, see Section \ref{sec_known}; it can be
effectively used for obtaining families of configurations,
see Section \ref{sec_admitExten} and Example
\ref{Ex_4_extension}(iii).

\item[(iv)] In Table 1, parameters of configurations
    $v_{n}^{\prime }$ with BDC incidence matrices are
    given. We use both (ii) and (iii) of Subsection
    \ref{subsec_bdc&families}. The starting weights
    $w_{i}^{\ast }$ are obtained by computer forming orbits
    of subgroups $\widehat{S}_{d}$ of a Singer group of
$PG(2,q)$. For $q=81$ we use (\ref
{eq3.3_Ex(ii)ProjPlane}). The values $k^{\prime },v^{\prime
}$ are calculated by
(\ref{eq3.1_rem(ii)}),(\ref{eq3.1_rem(iii)}). Only cases
with $ v^{\prime }<G(k^{\prime })$ are included in the
tables. Then the smallest value $k^{\#}$ for which
$v^{\prime }<G(k^{\#})$ is found. As a result, each row of
the table provides configurations $v_{n}^{\prime }$ with
$v^{\prime }<G(n),$ $n=k^{\#},k^{\#}+1,\ldots ,k^{\prime
},$ see (i) of Subsection \ref{subsec_bdc&families} and
(\ref{eq3.1_rem(i)}).
\end{description}

\begin{center}
INSERT Table 1 HERE
\end{center}

\begin{remark}
\label{rem3.3_Pepe} In \cite[Prop.\thinspace 3, Th.\thinspace
9]{Pepe}, parity check matrices of LDPC codes based on the
Hermitian curve in $ PG(2,q^{2}) $ and consisting of square
cyclic submatrices are constructed by geometrical tools that
can be considered as special cases of the more general approach
of Theorem \ref{Th3.3_the-same-meeting}. The mentioned parity
check matrices are incidence matrices of non-symmetric
configurations. It is possible that by dismissing some units in
the matrix,  BDC configurations could be obtained. This problem
is not considered here, nor in \cite{Pepe}. It is interesting
to note  that the matrix of \cite[Prop.\thinspace 3]{Pepe} uses
points belonging to the Hermitian curve,  whereas point set of
the symmetric configuration in \cite[Ex.\thinspace
3]{DGMP-GraphCodes}, whose parameters are as in (\ref
{eq2_ConstrA_GraphCode3}), coincides with the complement of the
same curve.
\end{remark}

Throughout the rest of the section,  cyclic starred affine
planes $(q^2-1)_Q$ are considered. In
\cite{DGMP-GraphCodes,LingDTSAP} useful results for $t$ a
divisor of $q+1$ are obtained. Theorem \ref {th3.3_AfPln} below
extends our knowledge on this topic and gives new results for
$t$ dividing $q-1.$ The proof is placed in Appendix; it uses
orbits of cyclic subgroup.

\begin{theorem}
\label{th3.3_AfPln}Let $q$ be an odd square.  Consider the
cyclic symmetric configuration $(q^2-q)_q$ associated to the
starred affine plane of order $q$. Let $t$ be a divisor of
$\sqrt q-1$, and let $d=(q^2-1)/t$. Let $\mathbf{A}$ be an
incidence BDC matrix of this configuration as in Theorem\emph{\
\ref{Th3.3_the-same-meeting}(ii)}. Let $ w_{0},w_{1},\ldots
,w_{t-1}$ be weights of the circulant $d\times d$ blocks of
$\mathbf{A}$.

\begin{description}
\item[(i)] Let $t=\sqrt{q}+1.$ Then $w_{0}=1,$ $
    w_{j}=\sqrt{q}-1$ for $j$ odd, $w_{j}=\sqrt{q}+1$ for
    $j$ even, $ j=1,2,\ldots ,\sqrt{q}$.

\item[(ii)] Let $t=\frac{1}{2}(\sqrt{q}+1),$
    $\sqrt{q}\equiv 1\pmod 4.$ Then $w_{0}=\sqrt{q},$
    $w_{1},w_{2},\ldots ,w_{t-1}=2 \sqrt{q}.$

\item[(iii)] Let $t=\frac{1}{2}(\sqrt{q}+1),$
    $\sqrt{q}\equiv 3\pmod 4.$ Then $w_{0}=\sqrt{q}+2,$
$w_{j}=2\sqrt{q}-2$ for $ j$ odd, $w_{j}=2\sqrt{q}+2$ for
$j$ even, $j=1,2,\ldots ,\frac{1}{2}(\sqrt{q} +1)-1$.

\item[(iv)] Let $t=\frac{1}{4}(\sqrt{q}+1),$
    $\sqrt{q}\equiv 3\pmod 4.$

\begin{itemize}
\item If $\frac{1}{4}(\sqrt{q}+1)$ is odd, then
    $w_{0}=3\sqrt{q},$ $ w_{1},w_{2},\ldots
    ,w_{t-1}=4\sqrt{q}.$

\item If $\frac{1}{4}(\sqrt{q}+1)$ is even, then
    $w_{0}=3\sqrt{q}+4,$ $ w_{j}=2\sqrt{q}-4$ for $j$
odd, $w_{j}=2\sqrt{q}+4$ for $j$ even, $ j=1,2,\ldots
,\frac{1}{4}(\sqrt{q}+1)-1$.
\end{itemize}
\end{description}
\end{theorem}

\begin{example}
\label{ex3.3_AfPlane} We consider the
the cyclic
symmetric configuration $(q^2-q)_q$ associated to the starred affine plane of order $q$.
 \cite{Bose},\cite[Sec.\thinspace 5.6]{Dimit},\cite[Th.\thinspace 19.17]{Shearer-Handb}, see
also \cite[Ex.\thinspace 5, Sec.\thinspace 6]{DGMP-GraphCodes}. In this case
the group $S$ is and affine Singer group of $AG(2,q)$.

\begin{description}
\item[(i)] Let $t$ be a divisor of $q+1.$ In
    \cite[Prop.\thinspace 10] {DGMP-GraphCodes} it is
proven that $w_{0}=(q+1)/t-1,$ $w_{1}=w_{2}=\ldots
=w_{t-1}=(q+1)/t,$ cf. Example \ref{ex_3.2_AffPlaneLing}
which uses results of \cite{LingDTSAP} obtained by a
different approach. Putting $t=q+1$ we obtain $ d=q-1$ and
$w_{0}=0,$ $w_{1}=w_{2}=\ldots =w_{q}=1$. Now using (ii) in
Subsection \ref{subsec_bdc&families} one can obtain a
family of configurations with
parameters~(\ref{eq2_affine_q-1}). Moreover, the case $v=1$
admits an extension by Procedure E, see Section
\ref{sec_known}; it can be effectively used for obtaining
families of configurations, see Section
\ref{sec_admitExten} and Example~\ref {Ex_4_extension}(i).

\item[(ii)] Let $q$ be an odd square. Let $t=\sqrt{q}+1,$
    $d=(\sqrt{q} -1)(q+1).$ By Theorem \ref{th3.3_AfPln}(i)
    and (iii) of Subsection \ref{subsec_bdc&families} one
    can obtain a family of configurations with parameters
\begin{equation}
v_{k}:v=2f(\sqrt{q}-1)(q+1),\text{ }k=(2f-1)\sqrt{q},\text{ }f=1,2,\ldots ,
\frac{\sqrt{q}+1}{2},\text{ }q\text{ odd square.}
\label{eq3.3_AfPlnPairsWeitghts}
\end{equation}
By using Theorem \ref{th3.3_AfPln}(ii),(iii),(iv), together
with (iii) of Subsection \ref{subsec_bdc&families}, we
obtain the same parameters as in (\ref
{eq3.3_AfPlnPairsWeitghts}). But the structure of an
incidence matrix $ \mathbf{M(}v,k\mathbf{)}$ is different
from that arising from Theorem \ref{th3.3_AfPln}(i).

\item[(iii)] In Table 2, parameters of configurations
    $v_{n}^{\prime }$ with BDC incidence matrices are
    given. We use (iii) of Subsection
    \ref{subsec_bdc&families}. The starting weights
    $w_{i}^{\ast }$ are obtained by computer through the
    constructions of the orbits of subgroups
    $\widehat{S}_{d}$ of the affine Singer group. Notations
    is as in Table~1.
\end{description}
\end{example}

\begin{center}
INSERT Table 2 HERE
\end{center}

\section{Constructing configurations $v_{k}$ admitting an extension\label
{sec_admitExten}} Establishing whether a configuration admits
an extension or not is  not an easy task in the general case.
In this section we deal with configurations admitting incidence
matrices with special type, and we show that may admit several
extensions.
\begin{definition}
\label{def3.2_structureE}Let $v=td,$ $t\geq k,$ $d\geq k-1$,
and let $v_k$ be a symmetric configuration. Let
$\mathbf{M(}v,k)$ be an incidence matrix of $v_k$, viewed as a
$t\times t$ block matrix, every block being of type $d\times
d$. We say that $\mathbf{M(}v,k)$ has \emph{Structure E} if
every $d\times d$ block is either a permutation matrix $
\mathbf{P}_{d}$ or the zero $d\times d$ matrix
$\mathbf{0}_{d}$.
\end{definition}

\begin{lemma}
\label{lem3.2_StructA}Let\textbf{\ }$v=td,t\geq k,d\geq k-1$,
and let $v_k$ be a  symmetric configuration. Assume that
$\mathbf{M(}v,k)$ is an incidence matrix of $v_k$ having
Structure E.

\begin{description}
\item[(i)] The matrix $\mathbf{M(}v,k)$ admits $\theta
    (t,d,k):=t\cdot \left\lfloor d/(k-1)\right\rfloor \geq
    t\geq k$ extensions.

\item[(ii)] Let $\mathbf{M(}v+\theta (t,d,k),k)$ be the
    matrix obtained from $(i)$ by applying $\theta(t,d,k)$
    extensions. Then $\mathbf{M(}v+\theta (t,d,k),k)$ above
    admits $\theta _{2}(t,d,k):=\left\lfloor \theta
    (t,d,k)/(k-1)\right\rfloor \geq 1$
    extensions.
    \end{description}
\end{lemma}

\begin{proof}

\begin{description}
\item[(i)] We need to provide $\theta (t,d,k)$ pairwise
    disjoint E-aggregates of $\mathbf{M(}v,k)$. For each
    block of type $\mathbf{P}_{d}$, one can easily define a
    set $\mathcal E$ of $\left\lfloor d/(k-1)\right\rfloor
    $ disjoint  E-aggregates consisting of row and columns
    with non-trivial intersection with $\mathbf{P}_{d}$.
Let $\mathbf B$ be the 01-matrix of type $t\times t$ such
that each entry corresponds to a $d\times d$ block in
$\mathbf M(v,k)$:  an entry is $1$ if the corresponding
block is of type $\mathbf P_d$, $0$ otherwise.  Each row
and each column of $\mathbf B$ has weight $k$. Therefore,
it is possible to obtain a permutation matrix
$\mathbf{P}_{t}$ by dismissing some units in $\mathbf B$.
The sets $\mathcal E$ defined from blocks $\mathbf{P}_{d}$
corresponding to units of $\mathbf{P}_{t}$ are clearly
disjoint, and their union gives $t\cdot \left\lfloor
d/(k-1)\right\rfloor $ not intersecting E-aggregates.

\item[(ii)] Theorem \ref{th2_thetaExten}(ii) can be applied
    $\theta_2(t,d,k)$ times.
\end{description}
\end{proof}

\begin{corollary}
\label{cor3.2_wh-in0,1} Let $v=td$, $t\geq k$, $d\geq k-1$, and
let $v_k$ be a symmetric configuration. Assume that an
incidence matrix $\mathbf A$ of  $v_{k}$ is a BDC matrix as in
\emph{(\ref{eq3.1 _block-circulant})} with weight vector
${\overline{\mathbf W}}(\mathbf A)=(w_0,\ldots,w_{t-1})$.

\begin{description}
\item[(i)] If all the weights $w_{u}$ belong to the set
    $\{0,1\}$, then $\mathbf A$ admits $t+1$ extensions.

\item[(ii)] If $\overline{\mathbf W}(\mathbf
    A)=(0,1,1,\ldots ,1)$, then one can obtain a family of
    symmetric configurations $v_{k}$ with parameters
\begin{equation}
v_{k}:v=cd+\theta ,\text{ }k=c-1-\delta ,\text{ }c=2,3,\ldots ,t,\text{ }
\theta =0,1,\ldots ,c+1,\text{ }\delta \geq 0\text{.}  \label{eq4_(011...1)}
\end{equation}

\item[(iii)] If $\overline{\mathbf W}(\mathbf A)=
    (1,1,\ldots ,1)$, then one can obtain a family of
symmetric configurations $v_{k}$ with parameters
\begin{equation}
v_{k}:v=cd+\theta ,\text{ }k=c-\delta ,\text{ }c=2,3,\ldots ,t,\text{ }
\theta =0,1,\ldots ,c+1,\text{ }\delta \geq 0\text{.}  \label{eq4_(11...1)}
\end{equation}
\end{description}
\end{corollary}

\begin{proof}
The matrix  $\mathbf{A}$ has clearly Structure E. Then (i)
follows  from Lemma \ref{lem3.2_StructA}, together with Theorem
\ref{th2_thetaExten}(i). As to (ii) and (iii), we use (ii) of
Subsection \ref{subsec_bdc&families}.
\end{proof}

In order to obtain a configuration having Structure E from  a
given  one, sometimes  the procedures described in Subsection
\ref{subsec_bdc&families} are useful; see Example
\ref{Ex_4_extension}(iii) below.

\begin{example}
\label{Ex_4_extension}

\begin{description}
\item[(i)] We consider the starred affine plane of order
    $q$ as a cyclic configuration $ (q^{2}-1)_{q}$, see
Examples \ref{ex_3.2_AffPlaneLing} and \ref
{ex3.3_AfPlane}(i). Let $t=q+1,$ $d=q-1,$ $w_{0}=0,$
$w_{1}=w_{2}=\ldots =w_{q}=1$. By
Corollary~\ref{cor3.2_wh-in0,1}(ii) we obtain a family of
symmetric configurations $v_{k}$ with parameters
\begin{equation}
v_{k}:v=c(q-1)+\theta ,\text{ }k=c-1-\delta ,\text{ }c=2,3,\ldots ,q+1,\text{
}\theta =0,1,\ldots ,c+1,\text{ }\delta \geq 0\text{.}  \label{eq4_AfPlExten}
\end{equation}

\item[(ii)] We consider Ruzsa's configuration
    $(p^{2}-p)_{p-1}$, see Example \ref{ex_3.2_Ruzsa}(ii).
    Put $t=p-1,$ $d=p.$ Then $w_{0}=w_{1}=\ldots
=w_{p-2}=1$. By Corollary~\ref{cor3.2_wh-in0,1}(iii) we
obtain a family with
\begin{equation}
v_{k}:v=cp+\theta ,\text{ }k=c-\delta ,\text{ }c=2,3,\ldots ,p-1,\text{ }
\theta =0,1,\ldots ,c+1,\text{ }\delta \geq 0,\text{ }p\text{ prime.}
\label{eq4_RuzExten}
\end{equation}
If $t=p,$ $d=p-1$ then $w_{0}=0,$ $w_{1}=\ldots
=w_{p-1}=1$. We obtain a family with
\begin{equation}
v_{k}:v=c(p-1)+\theta ,\text{ }k=c-1-\delta ,\text{ }c=2,3,\ldots ,p,\text{ }
\theta =0,1,\ldots ,c+1,\text{ }\delta \geq 0,\text{ }p\text{ prime.}
\label{eq4_RuzExten_d=p-1}
\end{equation}

\item[(iii)] Let $q$ be a square. We consider $PG(2,q)$ as
    a cyclic configuration $(q^{2}+q+1)_{q+1},$ see Example
    \ref{ex3.3_projPlane}(iii). Let $v=1,$
    $t=q-\sqrt{q}+1,$ $d=q+\sqrt{q}+1,$ $w_{0}=\sqrt{q}+1,$
    $ w_{1}=w_{2}=\ldots =w_{t-1}=1$. By (i) of Subsection
\ref{subsec_bdc&families} we can put $ w_{0}^{\prime }=1$
and obtain the weight vector$\mathbf{\ }(1,1,\ldots ,1).$
Now, by Corollary \ref{cor3.2_wh-in0,1}(iii), we obtain a
family with
\begin{eqnarray}
v_{k} &:&v=c(q+\sqrt{q}+1)+\theta ,\text{ }k=c-\delta ,\text{ }c=2,3,\ldots
,q-\sqrt{q}+1,  \notag \\
&&\theta =0,1,\ldots ,c+1,\text{ }\delta \geq 0\text{, }q\text{ square}
\label{eq4_BaerExten}
\end{eqnarray}
\end{description}
\end{example}

\section{The spectrum of parameters of cyclic symmetric configurations\label
{sec_ParamCyclicConfig}} In order to widen the ranges of
parameter pairs $\{v,k\}$ for which a cyclic symmetric
configuration $v_k$ exists,  we consider a number of procedures
that allow to define a new modular Golomb ruler from a known
one. Some methods have already been introduced in the paper,
see Theorem \ref{Th2-Mend}.

 Here we first recall a result from
\cite{Shearer-Handb},  which describes a method to construct
different rulers with the same parameters.

\begin{theorem}
\label{Th2_newMGRfromold} \emph{\cite{Shearer-Handb}} If $
(a_{1},a_{2},\ldots ,a_{k})$ is a $(v,k)$ modular Golomb ruler
and $m$ and $ b $ are integers with $\gcd (m,v)=1$ then
$(ma_{1}+b\pmod v,ma_{2}+b\pmod v ,\ldots ,ma_{k}+b\pmod v)$ is
also a $(v,k)$ modular Golomb ruler.
\end{theorem}

It should be noted that a $(v,k)$ modular Golomb ruler can be a
$(v+\Delta,k)$ modular Golomb ruler for some integer $\Delta$
\cite{Gropp-nk}.  This property does not depend  on parameters
$v$ and $k$ only. This is why Theorem \ref{Th2_newMGRfromold}
can be useful for our purposes.

\begin{example}
\label{ex5_extenGR} We consider the $(31,6)$ modular Golomb
ruler $$(a_{1},\ldots ,a_{6})=(0,1,4,10,12,17)$$ obtained from
$PG(2,5),$ see \cite{ShearerWebShortest}. We can apply Theorem
\ref{Th2_newMGRfromold} for $m=19$, $b=0$. The $(31,6)$ modular
Golomb ruler $(ma_{1}$ $(\text{ mod }$ $ 31),\ldots ,ma_{6}$
$(\text{ mod }31))$ is $$(a_{1}^{\prime },\ldots ,a_{6}^{\prime
})=(0,4,11,13,14,19).$$ Now we take $\Delta =4$ and calculate
the set of differences $\{a_{i}^{\prime }-a_{j}^{\prime
}\,(\text{ mod }35)| \,1\leq i,j\leq 6;i\neq j\}$, that is
$\{1,2,3,4,5,6,7,8,9,10,11,13,14,15,16,19,20,
21,22,24,25,26,27,28,29,30,\linebreak 31,32,33,34\}.$ As the
all differences are distinct and nonzero, the starting $(31,6)$
modular Golomb ruler is also a $(35,6)$ modular Golomb ruler.
\end{example}

\begin{example}
We take the $(57,8)$ modular Golomb ruler $(a_{1},\ldots
,a_{8})=(0,4,5,17,19,25,\linebreak 28,35)$ obtained from
$PG(2,7)$. Apply Theorem  \ref{Th2-Mend} for $\delta =1$, and
remove the integer $35$. A $(57,7)$ modular Golomb ruler
$(a_{1}^{\prime },\ldots ,a_{7}^{\prime })=(0,4,5,17,19,25,28)$
is obtained. Now we take $\Delta =-2.$ Due to Definition
\ref{def1_Sidon} and Theorem \ref {th1_Sidon=Golomb}, instead
of differences we calculate the set of sums $ \{a_{i}^{\prime
}+a_{j}^{\prime }\,(\text{ mod }55)|\,1\leq i\leq j\leq 7\}$,
that is $\{0,4,5,8,9,10,17,19,21,22,23,24,25,28,29,30,
32,33,34,36,38,42,44,45,47,50,53,56\}.$ As the all sums are
distinct, the $ (57,7)$ modular Golomb ruler
$(0,4,5,17,19,25,28)$ is also a $(55,7)$ modular Golomb ruler.
\end{example}

For $k\le 81$, we performed a computer search starting from the
$(v,k)$ modular Golomb rulers corresponding to
(\ref{eq2_cyclicPG(2,q)})--(\ref{eq2_cyclicRuzsa}). For
projective and affine planes, we got a concrete description of
the ruler from \cite{ShearerWebShortest}. For Ruzsa's
construction, we used (\ref{eq3.2_Ruzsa}) of
Example~\ref{ex_3.2_Ruzsa}. For every starting $(v,k)$ modular
ruler we first considered all possible $m$ with $\gcd (m,v)=1$,
and applied Theorem \ref{Th2_newMGRfromold} for $b=0$ to get
new rulers with the same parameters $v$ and $k$. Then, we
checked whether this ruler was also a $(v+\Delta,k)$ for some
$\Delta$.

We obtained improvements on the known results for $k\ge 16$.
For the sake of completeness, we summarize the known results
about the case $k\le 15$ in Table 3. Table 4 lists known and
new results  about $16\le k \le 41$, whereas Table 5 deals with
the case $42\le k\le 83$.

 As to Table3, the values of $v_{\delta
}(k)$ are taken from \cite[Tab.\thinspace
IV]{GrahSloan},\cite[Tab.\thinspace 2] {OstergModulSidon},
\cite[Tab.\thinspace 1a]{ShearerWebModulGR}, \cite{Swanson}.
The values of $v$ for which cyclic symmetric configurations
$v_{k}$ exist (resp. do not exist) are written in normal (resp.
in italic) font. Moreover, $\overline{v}$ means that no
configuration $v_{k}$ exists while $\overline{ v^{c}}$ notes
that no cyclic configuration $v_{k}$ exists. Data from
\cite{Funk2008,Gropp-nk,Gropp-nonsim,Gropp-Handb,KaskiOst,Lipman,ShearerWebModulGR}
are listed in the 4-th column of the table. We take into
account that an entry of the form \textquotedblleft
$t+$\textquotedblright\ in the row \textquotedblleft
$n$\textquotedblright\ of \cite[Tab.\thinspace
1]{ShearerWebModulGR} means the existence of a cyclic symmetric
configurations $v_{n}$ with $v\geq t.$ Also, we use the
following \emph{non-existence} results: $\overline{\emph{32}
_{6}}$ \cite[Th.\thinspace 4.8]{Gropp-nk};
$\overline{\emph{33}_{6}}$ \cite {KaskiOst};
$\overline{\emph{34}_{6}^{c}},$
$\overline{\emph{59}_{8}^{c}}$-$ \overline{\emph{62}_{8}^{c}}$
\cite{Lipman}; $\overline{\emph{75}_{9}^{c}}$-$
\overline{\emph{79}_{9}^{c}},\overline{\emph{81}_{9}^{c}}$-$\overline{\emph{84}_{9}^{c}}$
\cite{Funk2008}. The following theorem of \cite{Gropp-nonsim}
is taken into account.
\begin{theorem}
\label{th5_deficiency1}\emph{\cite[Th.\thinspace
2.4]{Gropp-nonsim}} There is no symmetric configuration
$(k^{2}-k+2)_{k}$ if $\,5\leq k\leq 10$ or if neither $k$ or
$k-2$ is a square.
\end{theorem}
\noindent The values of $k$ for which the spectrum of
parameters of cyclic symmetric configurations $v_{k}$ is
completely known are indicated by a dot "$\centerdot $"; the
corresponding values of $E_{c}(k)$ are sharp and they are noted
by the dot "$\centerdot $" too. Values of $v$ obtained by our
search lies within the range of the known parameters (see the
5-th column of the table).

\begin{center}
INSERT Table 3  HERE
\end{center}

 In Table 4, for $16\leq k\leq 41,$ $P(k)\leq v<G(k),$
data on the existence of cyclic symmetric configuration $v_{k}\
$are given. The known results from
(\ref{eq2_cyclicPG(2,q)})--(\ref{eq2_cyclicRuzsa}) are written
in normal font; the entries $v_{a},v_{b},$ and $v_{c}$ means,
respectively, that the relations
(\ref{eq2_cyclicPG(2,q)}),(\ref{eq2_cyclicAG(2,q)}), and (\ref
{eq2_cyclicRuzsa}) are used. The values of $v$ obtained in this
work are given in bold font; the entry
$\mathbf{v}$-$\mathbf{w}$ notes an interval of sizes from
$\mathbf{v}$ to $\mathbf{w}$ without gaps. If an already known
value lies within an interval $\mathbf{v}$-$\mathbf{w}$
obtained in this work, then it is  written immediately before
the interval. Also, some data on the nonexistence (including
those arising from Theorem \ref{th5_deficiency1}) are written
in italic font, in the form $\overline{v}$ or
$\overline{v^{c}}$. For $k=16$ the value $v_{\delta }(k)=255$
\cite{ShearerWebModulGR} is taken into account. The
nonexistence of some projective planes by Bruck-Ryser theorem
is  also indicated.

\begin{center}
INSERT Table 4 HERE
\end{center}

In Table 5, for $42\leq k\leq 83,$ the upper bounds on the
cyclic existence bound $E_{c}(k)$ obtained in this work are
listed.

 \begin{center}
INSERT Table 5 HERE
\end{center}

\section{The spectrum of parameters of symmetric (non-necessarily cyclic) configurations  \label{sec6_spectrum}}

The known results regarding to parameters of symmetric
configurations can be found in
\cite{AFLN-graphs}--\cite{Baker},\cite{Boben-v3}--\cite
{CarsDinStef-Reduc-n3},\cite{DGMP-ACCT2008}--\cite{DGMP-GraphCodes},\cite{Funk2008,FunkLabNap,GH,KaskiOst,Ruzsa,Shearer-Handb,ShearerWebModulGR},\cite{GrahSloan}--\cite{OstergModulSidon},\cite{Lipman}--\cite{MePaWolk};
see also the references therein.

The known families of configurations $v_{k}$ were described in
Section \ref {sec_known}. In Table~6, for $k\leq 37,$ $P(k)\leq
v<G(k),$ values of $v$ for which a symmetric configuration
$v_{k}$ from one of
 the families of Section \ref{sec_known} exists are given. A subscript of an entry
 indicates that a specific (2.$i$) is used: more precisely $
v_{a}$ indicates that $v$ is obtained from
($\ref{eq2_cyclicPG(2,q)}$), and similarly $v_{b}\rightarrow
($\ref {eq2_cyclicAG(2,q)}$),$ $v_{c}\rightarrow
($\ref{eq2_cyclicRuzsa}$),$ $ v_{d}\rightarrow
($\ref{eq2_q-cancel}$),$ $v_{e}\rightarrow ($\ref
{eq2_q-1-cancel}$),$ $v_{f}\rightarrow ($\ref{eq2_Baer}$),$ $
v_{g}\rightarrow ($\ref{eq2_Baer2}$),$ $v_{h}\rightarrow ($\ref
{eq2_FuLabNabDecomp}$),$ $v_{i}\rightarrow
($\ref{eq2_affine_q-1}$),$ $ v_{j}\rightarrow
($\ref{eq2_ConstrA_GraphCode1}$),$ $v_{k}\rightarrow ($\ref
{eq2_ConstrA_GraphCode2}$),$ $v_{l}\rightarrow
($\ref{eq2_ConstrA_GraphCode3} $),$ $v_{m}\rightarrow
($\ref{eq2_tetaExten}$).$ An entry with more than one subscript
means that the same value  can be obtained from  different
constructions. An entry of type
$v_{\text{subscript}_1,\text{subscript}_2,\ldots}-v'_{\text{subscript}_1,\text{subscript}_2,\ldots}$
indicates that a whole interval of values from $v$ to $v'$ can
be obtained from the constructions corresponding to the
subscripts.

To save space, in Table 6 if a value belongs to an interval
obtained by the Extension Construction of
(\ref{eq2_tetaExten}), then it is listed only once, even if it
can be obtained from different constructions as well.

\begin{center}
INSERT Table 6 HERE
\end{center}

In Table 7, for $P(k)\leq v<G(k),$ parameters of the symmetric
configurations $v_{k}$  from Sections \ref {sec_doub-circ} and
\ref{sec_admitExten} are listed. An entry of type $
v_{\text{subscript}}$ indicates that either relations
(3.$i$),(4.$j$) or Tables 1, 2 are used. More precisely
$v_{n}\rightarrow ($\ref{eq3.3_Ex(ii)ProjPlane}$ ) $,
$v_{p}\rightarrow ($\ref{eq3.3_AfPlnPairsWeitghts}$)$, $
v_{r}\rightarrow ( $\ref{eq4_AfPlExten}$)$, $v_{s}\rightarrow
($\ref {eq4_RuzExten}$)$, $v_{t}\rightarrow
($\ref{eq4_RuzExten_d=p-1}$)$, $ v_{u}\rightarrow
($\ref{eq4_BaerExten}$),$ $v_{v}\rightarrow $ Table 1, $
v_{w}\rightarrow $ Table 2. For $k\leq 37,$ we listed all the
results we got, whereas for $k=38$-$41,49,56$ we only give some
illustrative examples.

\begin{center}
INSERT Table 7 HERE
\end{center}

We note that a number of parameters are new: $322_{16},$
$458_{19},$ $459_{19},482_{20},574_{22},$ $ 674_{24},$
$782_{26},$ $1066_{27}$-$1072_{27},$ $1104_{27}$-$1106_{27},$ $
1066_{28}$-$1072_{28},$ $1104_{28}$-$1109_{28},$
$1142_{28}$-$1146_{28},$ $ 1104_{29}$-$1109_{29},$
$1142_{29}$-$1146_{29},1180_{29}$-$1183_{29},$ $ 1220_{29},$
$1142_{30}$-$1146_{30},$ $1180_{30}$-$1183_{30},$ $1218_{30}$-$
1220_{30},$ $1180_{31}$-$1183_{31},$ $1218_{31}$-$1220_{31},$
$1256_{31}$, $ 1257_{31},$ $1218_{32}$-$1220_{32},$
$1256_{32}$-$1257_{32},1294_{32},$ $ 1256_{33}$, $1257_{33},$
$1294_{33},$ $1294_{34},$ $1430_{34}$-$1434_{34},$ $
1472_{35}$-$1475_{35},$ $1514_{36}$-$1516_{36},$ $1556_{37},$
$1557_{37};$ sometimes  the gaps  in an interval arising from
(\ref{eq2_tetaExten}) are filled.

The new cyclic configurations from Table 4, like
$382_{17}$-$390_{17},$ $ 401_{18},$ $405_{18}$-$407_{18},$
$410_{18},$ $412_{18}$, $413_{18},$ also fill some gaps in the
known range of parameters.

Parameters of the family of Example \ref{ex3.3_projPlane}(i)
are too big to be included  in Table 7. For the same reason,
parameters for (\ref{eq3.3_Ex(ii)ProjPlane}) are only reported
for $q=3^{4}$.

Finally, in Table 8, for $k\leq 37,$ $P(k)\leq v<G(k),$ we
summarize the data from Tables 3,4,6, and 7. Also, we use the
following known results on existence of sporadic symmetric
configurations: $45_{7}$ \cite {Baker}; $82_{9}$
\cite[Tab.\thinspace 1]{FunkLabNap}; $135_{12},$ see \cite
{Gropp-nonsim} with reference to Mathon's talk at the British
Combinatorial Conference 1987; $34_{6}$ \cite{Krcadinac}. The
non-existence of configuration $112_{11}$ is proved in
\cite{KaskOst112-11}.

\begin{center}
INSERT Table 8 HERE
\end{center}

\noindent In Table 8, the values of $k$ for which the spectrum
of parameters of symmetric configurations $v_{k}$ is completely
known are indicated by a dot "$\centerdot $"; the corresponding
values of $E(k)$ are exact and they are indidcated by a dot as
well. The filling of the interval $P(k)$--$ G(k)$ is
 expressed as a percentage in the last column.
It is interesting to note that such a percentage is quite high,
and that most gaps occur for $v$ close to $k^2-k+1$.

\section*{Appendix: Proof of Theorem \ref{th3.3_AfPln}}

Let $\xi $ be a primitive element of $F_{q^{2}}.$ Let $\omega
=\xi ^{\frac{q+1 }{2}}$, and $\theta =\omega ^{q+1}$. Identify
a point $(x,y)\in AG(2,q)$ with the element $z=x+\omega y\in
F_{q^{2}}$. As $\omega ^{q-1}=-1$ it is straightforward to
check that $z^{q+1}=x^{2}+\theta y^{2}$.

We need to consider the
orbits of $F_{q^{2}}^{\ast }$ under the action of the cyclic group generated
by $\sigma ^{t}$ where $\sigma (\xi ^{i})=\xi ^{i+1}$. The orbit $O_{j}$ of
the element $\xi ^{j}$ is $\{\xi ^{j},\xi ^{j+t},\xi ^{j+2t},\ldots ,\xi
^{j+(\frac{q^{2}-1}{t}-1)t}\}$. Let $\mu $ be a primitive element in $F_{q}$.

(i) Assume that $t=\sqrt{q}+1.$ For each $z\in O_{j}$ we have
$z^{q+1}=\xi ^{j(q+1)}(\xi ^{(q+1)(\sqrt{q} +1)})^{h}$ for some
$h$. Also, $\xi ^{q+1}$ is a primitive element of $F_{q}$ and
$(\xi ^{(q+1)(\sqrt{q}+1)})^{h}\in F_{\sqrt{q}}^{\ast }$. This
means for each $j=0,\ldots ,\sqrt{q}$, the orbit $O_{j}$
consists precisely of the elements $z$ such that $z^{q+1}\in
\mu ^{j}F_{\sqrt{q}}^{\ast }$. Therefore, the following lemma
holds.

\begin{lemma}
The orbit $O_{j}$ in $AG(2,q)$ consists of the union of the $\sqrt{q}-1$
conics with equation $x^{2}+\theta y^{2}=\mu ^{j}\alpha $ where $\alpha \in
F_{\sqrt{q}}^{\ast }.$
\end{lemma}

The final step is to compute the sizes of the intersections
$|O_{j}\cap \ell |$ where $\ell $ is any line of $AG(2,q)$ not
passing through the origin. These sizes are the integers
$w_{0},\ldots ,w_{\sqrt{q}}$. Choose the line $ \ell :x=1$.

\begin{lemma}
$w_{0}=1$.
\end{lemma}

\begin{proof}
We prove that $|\ell \cap O_{0}|=1.$ Note that $O_{0}$ consists
of the conics $C_{\alpha }:x^{2}+\theta y^{2}=\alpha .$ The
line $x=1$ meets the conic $C_{\alpha }$ in one point if
$\alpha =1$. If $\alpha \neq 1$ then the intersection is empty
since $\theta $ is not a square in $F_{q}$ (and $ \alpha $ is a
square since it is an element of $F_{\sqrt{q}}$).
\end{proof}

\begin{lemma}
Let $N_{j}$ be the number of non-squares in the set $\mu
^{j}F_{\sqrt{q} }^{\ast }-1$. Then $w_{j}=2N_{j}.$
\end{lemma}

\begin{proof}
The conic $C_{\alpha }:x^{2}+\theta y^{2}=\mu ^{j}\alpha $
meets the line $ x=1$ in $0$ or $2$ points. The latter case
occurs precisely when $\mu ^{j}\alpha -1$ is not a square.
\end{proof}

In order to compute the integers $N_{j}$, the following lemmas will be
useful.

\begin{lemma}
The collection of sets $H_{\beta }=\{\frac{\mu }{1-\mu \beta
}F_{\sqrt{q} }^{\ast }-1\mid \beta \in F_{\sqrt{q}}\}$
coincides with $\{\mu ^{j}F_{\sqrt{ q}}^{\ast }-1\mid
j=1,\ldots ,\sqrt{q}\}$.
\end{lemma}

\begin{proof}
We need only need to show that sets $\frac{\mu }{1-\mu \beta
}F_{\sqrt{q} }^{\ast }$ are pairwise distinct. Assume on the
contrary that $\frac{\mu }{ 1-\mu \beta }=\alpha \frac{\mu
}{1-\mu \gamma }$ for $\beta ,\gamma \in F_{ \sqrt{q}}$,
$\alpha \in F_{\sqrt{q}}^{\ast }$. Then $\alpha (1-\mu \beta
)=1-\mu \gamma $ that is $\mu (-\alpha \beta +\gamma )=1-\alpha
.$ If $ \alpha \beta -\gamma =0$ then $\alpha =1$ and hence
$\beta =\gamma $. If $ \alpha \beta -\gamma \neq 0$ then $\mu
\in F_{\sqrt{q}}$, which is a contradiction.
\end{proof}

\begin{lemma}
\label{key} Fix $\beta \in F_{\sqrt{q}}$ and $j\in \{1,\ldots ,\sqrt{q}\}$.
Assume that the set $\frac{\mu }{1-\mu \beta }F_{\sqrt{q}}^{\ast }-1$
coincides with $\mu ^{j}F_{\sqrt{q}}^{\ast }-1$. Then $1-\mu \beta $ is a
square if and only if $j$ is odd.
\end{lemma}

\begin{proof}
Note that $\frac{\mu }{1-\mu \beta }F_{\sqrt{q}}^{\ast }-1$
coincides with $ \mu ^{j}F_{\sqrt{q}}^{\ast }-1$ if and only if
$\mu ^{j-1}(1-\mu \beta )\in F_{\sqrt{q}}$. Then $\mu
^{j-1}(1-\mu \beta )$ is a square. Whence the assertion
follows.
\end{proof}

Let $M_{\beta }$ be the number non-squares in the set
$\frac{\mu }{1-\mu \beta }F_{\sqrt{q}}^{\ast }-1$. By the
previous lemma, the set of integers $ M_{\beta }$ coincides
with the set of integers $N_{j}$. Let $A=M_{0}=N_{1}$. Next we
show that every $M_{\beta }$ is related to $A$.

\begin{lemma}
\label{key2} If $1-\mu \beta $ is a square in $F_{q}$ then $M_{\beta }=A$.
If $1-\mu \beta $ is not a square in $F_{q}$ then $M_{\beta }=\sqrt{q}-A$.
\end{lemma}

\begin{proof}
$A$ is the number of non-squares in the set $H_{0}=\{\mu \alpha
-1\mid \alpha \in F_{\sqrt{q}}^{\ast }\}$. For each $\beta \in
F_{\sqrt{q}}$, this set coincides with $\{\mu (\alpha +\beta
)-1\mid \alpha \in F_{\sqrt{q} }^{\ast },\alpha \neq -\beta
\}\cup \{\mu \beta -1\}.$ But since $\mu (\alpha +\beta )-1=\mu
\alpha +\mu \beta -1=(1-\mu \beta )(\frac{\mu }{1-\mu \beta
}\alpha -1)$ we have that
\begin{equation*}
H_{0}=(1-\mu \beta )\{\frac{\mu }{1-\mu \beta }\alpha -1\mid \alpha \in F_{
\sqrt{q}}^{\ast },\alpha \neq -\beta \}\cup \{\mu \beta -1\},
\end{equation*}
that is $H_{0}=(1-\mu \beta )H_{\beta }\setminus \{-1\}\cup \{\mu \beta
-1\}. $ Two cases have to be distinguished.

\emph{a) }$1-\mu \beta $ is a square. Then either $\frac{1}{\mu
\beta -1}$ and $\mu \beta -1$ are both squares or are both
non-squares. It follows that the number of non-squares in
$H_{0}$ equals the number of non-squares in $ H_{\beta }$.
Therefore, $M_{\beta }=A$.

\emph{b) }$1-\mu \beta $ is not a square. As $-1$ is a square,
$\mu \beta -1$ is not a square. The number of non-squares in
$H_{0}$ equals the number of squares in $H_{\beta }$ plus $1$.
Then $M_{\beta }=\sqrt{q}-A$.
\end{proof}

As a corollary to Lemmas \ref{key} and \ref{key2}, the following result is
obtained.

\begin{lemma}
\label{key3} If $j$ is odd then $N_{j}=A$. If $j$ is even then
$N_{j}=\sqrt{q }-A$.
\end{lemma}

By the above lemma \emph{only two possibilities occur for
$w_{j}$, namely $ 2A $ and $2(\sqrt{q}-A)$.} Next we calculate
$A$. Let $u_{1}$ be the number of $\beta ^{\prime }s$ such that
$1-\mu \beta $ is a square in $GF(q)$. Then, from
$w_{0}+w_{1}+\ldots +w_{\sqrt{q}}=q$, we obtain
\begin{equation*}
q=1+2u_{1}A+2(\sqrt{q}-u_{1})(\sqrt{q}-A)=1+4u_{1}A+2q-2\sqrt{q}(u_{1}+A).
\end{equation*}
Hence,
\begin{equation*}
0=q+1+4u_{1}A-2\sqrt{q}(u_{1}+A)=(\sqrt{q}-2u_{1})(\sqrt{q}-2A)+1
\end{equation*}
Since both $\sqrt{q}-2u_{1}$ and $\sqrt{q}-2A$ are integers,
the only possibility is that they are both equal to $\pm 1$.
This implies that $A= \frac{\sqrt{q}\pm 1}{2}$,
$u_{1}=\frac{\sqrt{q}\mp 1}{2}$ is the only solution. So, we
have proved that the integers $w_{0},\ldots ,w_{\sqrt{q}}$ are
such that: $1$ occurs precisely once; the integer $\sqrt{q}+1$
occurs $\frac{ \sqrt{q}-1}{2}$ times; the integer $\sqrt{q}-1$
occurs $\frac{\sqrt{q}+1}{2}$ times. Finally, since
$w_{j}=w_{j^{\prime }}$ if $j=j^{\prime }\pmod 2$ and the
number of odd integers in $[1,\sqrt{q}]$ is greater than that
of even integers, the assertion of Theorem \ref{th3.3_AfPln}(i)
follows.

(ii) Assume that  $t=\frac{1}{2}(\sqrt{q}+1),$ $\sqrt{q}\equiv
1\pmod 4$. An orbit here is the union of two orbits $O_{j}$ of
case (i). More precisely, an orbit consists of the union of the
$2( \sqrt{q}-1)$ conics with equation
\begin{equation*}
x^{2}+\theta y^{2}=\mu ^{j}\alpha ,\qquad \alpha \in F_{\sqrt{q}}^{\ast
}\cup \mu ^{\frac{\sqrt{q}+1}{2}}F_{\sqrt{q}}^{\ast }.
\end{equation*}
Equivalently, an orbit here is the union $O_{j}\cup
O_{j+\frac{\sqrt{q}+1}{2} }$ for some $j=0,\ldots
,\frac{\sqrt{q}-1}{2}$. Assume that $j>0$. Since $j$ and
$j+\frac{\sqrt{q}+1}{2}$ are different modulo 2, we have that
the number of points of $\ell _{i}$ in this orbit is
$(\sqrt{q}-1)+(\sqrt{q}+1)=2\sqrt{q }$. If $j=0$, since
$\frac{\sqrt{q}+1}{2}$ is odd we have that $\ell _{1}$ meets
$O_{0}\cup O_{\frac{\sqrt{q}+1}{2}}$ in $\sqrt{q}$ points.

(iii) Assume that $t=\frac{1}{2}(\sqrt{q}+1),$ $\sqrt{q}\equiv
3\pmod 4$. Again, an orbit here is the union $O_{j}\cup
O_{j+\frac{\sqrt{q}+1}{2}}$ for some $j=0,\ldots
,\frac{\sqrt{q}-1}{2}$. Assume that $j>0$. Since $j$ and $j+
\frac{\sqrt{q}+1}{2}$ are equal modulo 2, we have that the
number of points of $\ell _{i}$ in this orbit is $2\sqrt{q}-2$
if $j$ is odd, $2\sqrt{q}+2$ if $j$ is even. If $j=0$, since
$\frac{\sqrt{q}+1}{2}$ is even we have that $ \ell _{1}$ meets
$O_{0}\cup O_{\frac{\sqrt{q}+1}{2}}$ in $\sqrt{q}+2$ points.

(iv) Assume that  $t=\frac{1}{4}(\sqrt{q}+1),$ $\sqrt{q}\equiv
3\pmod 4$. An orbit here is the union $O_{j}\cup
O_{j+\frac{\sqrt{q}+1}{4}}\cup O_{j+ \frac{\sqrt{q}+1}{2}}\cup
O_{j+3\frac{\sqrt{q}+1}{4}}$, for some $j=0,\ldots
,\frac{\sqrt{q}-3}{4}$. Then it is easy to deduce the assertion
of Theorem \ref{th3.3_AfPln}(iv). $ \hfill \square $

\newpage
\begin{center}
\textbf{Table 1. }Parametrs of configurations $v_{n}^{\prime }$
with BDC incidence matrices$\mathbf{,}$ $v^{\prime }<G(n),$
$n=k^{\#},k^{\#}+1,\ldots ,k^{\prime },$ by (ii) of Subsection
\ref{subsec_bdc&families} from the cyclic projective plane $
PG(2,q)$

$
\begin{array}{r|c|r|c|c|r|r|r|r|r}
\hline
q & t & d & w_{i}^{\ast } & c & k^{\prime } & v^{\prime } & G(k^{\prime }) &
k^{\#} & G(k^{\#}) \\ \hline
32 & 7 & 151 & 0,5,5,6,5,6,6 & 6 & 25 & 906 & 961 & 25 & 961 \\ \hline
37 & 3 & 469 & 16,9,13 & 2 & 25 & 938 & 961 & 25 & 961 \\ \hline
43 & 3 & 631 & 19,13,12 & 2 & 31 & 1262 & 1495 & 30 & 1361 \\ \hline
49 & 3 & 817 & 21,16,13 & 2 & 34 & 1634 & 1877 & 33 & 1719 \\ \hline
61 & 3 & 1261 & 25,21,16 & 2 & 41 & 2522 & 2611 & 40 & 2565 \\ \hline
64 & 3 & 1387 & 27,19,19 & 2 & 46 & 2774 & 3407 & 42 & 2795 \\ \hline
67 & 3 & 1519 & 28,19,21 & 2 & 47 & 3038 & 3609 & 44 & 3193 \\ \hline
73 & 3 & 1801 & 28,27,19 & 2 & 47 & 3602 & 3609 & 47 & 3609 \\ \hline
79 & 3 & 2107 & 31,21,28 & 2 & 52 & 4214 & 4541 & 51 & 4381 \\ \hline
81 & 7 & 949 & 4,13,13,13,13,13,13 & 6 & 69 & 5694 & 8291 & 58 & 5703 \\
\hline
81 & 7 & 949 & 4,13,13,13,13,13,13 & 5 & 56 & 4745 & 5451 & 54 & 4747 \\
\hline
97 & 3 & 3169 & 39,28,31 & 2 & 67 & 6338 & 7639 & 62 & 6431 \\ \hline
103 & 3 & 3571 & 39,28,37 & 2 & 67 & 7142 & 7639 & 65 & 7187 \\ \hline
107 & 7 & 1651 & 24,15,15,13,15,13,13 & 6 & 89 & 9906 & 13557 & 75 & 9965 \\
\hline
107 & 7 & 1651 & 24,15,15,13,15,13,13 & 5 & 76 & 8255 & 10179 & 69 & 8291 \\
\hline
109 & 3 & 3997 & 43,36,31 & 2 & 74 & 7994 & 9507 & 69 & 8291 \\ \hline
109 & 7 & 1713 & 8,15,15,19,15,19,19 & 6 & 83 & 10278 & 12041 & 77 & 10409
\\ \hline
121 & 7 & 2109 & 21,20,13,13,21,13,21 & 6 & 86 & 12654 & 13075 & 85 & 12821
\\ \hline
127 & 3 & 5419 & 49,43,36 & 2 & 85 & 10838 & 12821 & 80 & 11127 \\ \hline
128 & 7 & 2359 & 24,21,21,14,21,14,14 & 6 & 94 & 14154 & 15769 & 91 & 15085
\\ \hline
137 & 7 & 2701 & 24,15,15,23,15,23,23 & 6 & 99 & 16206 & 17081 & 96 & 16243
\\ \hline
139 & 3 & 6487 & 52,39,49 & 2 & 91 & 12974 & 15085 & 86 & 13075 \\ \hline
151 & 3 & 7651 & 57,43,52 & 2 & 100 & 15302 & 17663 & 93 & 15453 \\ \hline
163 & 3 & 8911 & 63,49,52 & 2 & 112 & 17822 & 27043 & 102 & 18437 \\ \hline
151 & 7 & 3279 & 32,19,19,21,19,21,21 & 6 & 127 & 19674 & 28921 & 105 & 19769
\\ \hline
\end{array}
$\newpage

\textbf{Table 2. }Parametrs of configurations $v_{n}^{\prime }$
with BDC incidence matrices$\mathbf{,}$ $v^{\prime }<G(n),$
$n=k^{\#},k^{\#}+1,\ldots ,k^{\prime },$ by (ii) and (iii) of
Subsection  \ref{subsec_bdc&families} from the cyclic affine
plane $AG(2,q)$

$\renewcommand{\arraystretch}{0.85}
\begin{array}{r|c|r|c|c|r|r|r|r|r|r}
\hline
q & t & d & w_{i}^{\ast } & c & f & k^{\prime } & v^{\prime } & G(k^{\prime
}) & k^{\#} & G(k^{\#}) \\ \hline
31 & 3 & 320 & 14,9,8 & 2 &  & 22 & 640 & 713 & 21 & 667 \\ \hline
37 & 3 & 456 & 16,9,12 & 2 &  & 25 & 912 & 961 & 25 & 961 \\ \hline
49 & 4 & 600 & 16,12,9,12 & 3 &  & 34 & 1800 & 1877 & 34 & 1877 \\ \hline
49 & 6 & 400 & 4,9,12,8,8,8 & 5 &  & 36 & 2000 & 2011 & 36 & 2011 \\ \hline
53 & 4 & 702 & 17,12,10,14 & 3 &  & 37 & 2106 & 2199 & 37 & 2199 \\ \hline
61 & 3 & 1240 & 25,16,20 & 2 &  & 41 & 2480 & 2611 & 39 & 2505 \\ \hline
67 & 3 & 1496 & 26,24,17 & 2 &  & 43 & 2992 & 3015 & 43 & 3015 \\ \hline
71 & 5 & 1008 & 8,14,15,16,18 & 4 &  & 50 & 4032 & 4189 & 50 & 4189 \\ \hline
73 & 3 & 1776 & 30,21,22 & 2 &  & 51 & 3552 & 4381 & 47 & 3609 \\ \hline
79 & 3 & 2080 & 32,25,22 & 2 &  & 54 & 4160 & 4747 & 50 & 4189 \\ \hline
79 & 6 & 1040 & 8,14,13,14,18,12 & 5 &  & 56 & 5200 & 5451 & 56 & 5451 \\
\hline
79 & 6 & 1040 & 18,12,8,14,13,14 &  & 2 & 50 & 4160 & 4189 & 50 & 4189 \\
\hline
81 & 4 & 1640 & 25,20,16,20 & 3 &  & 57 & 4920 & 5547 & 55 & 5197 \\ \hline
81 & 4 & 1640 & 25,20,16,20 &  & 1 & 45 & 3280 & 3375 & 45 & 3375 \\ \hline
81 & 8 & 820 & 16,8,8,8,9,12,8,12 & 7 &  & 64 & 5740 & 7055 & 59 & 5823 \\
\hline
81 & 8 & 820 & 16,8,8,8,9,12,8,12 & 6 &  & 56 & 4920 & 5451 & 55 & 5187 \\
\hline
83 & 8 & 861 & 6,10,10,10,15,11,10,11 & 7 &  & 66 & 6027 & 7515 & 60 & 6039
\\ \hline
83 & 8 & 861 & 6,10,10,10,15,11,10,11 & 6 &  & 56 & 5166 & 5451 & 55 & 5187
\\ \hline
89 & 4 & 1980 & 26,25,18,20 & 3 &  & 62 & 5940 & 6431 & 60 & 6039 \\ \hline
89 & 8 & 990 & 17,8,10,12,8,10,10,14 & 7 &  & 65 & 6930 & 7187 & 64 & 7055
\\ \hline
97 & 3 & 3136 & 37,34,26 & 2 &  & 63 & 6272 & 6783 & 62 & 6431 \\ \hline
97 & 4 & 2352 & 29,26,20,22 & 3 &  & 69 & 7056 & 8291 & 65 & 7187 \\ \hline
97 & 6 & 1586 & 21,20,14,16,14,12 & 5 &  & 69 & 7930 & 8291 & 69 & 8291 \\
\hline
101 & 4 & 2550 & 30,25,20,26 & 3 &  & 70 & 7650 & 8435 & 68 & 7913 \\ \hline
101 & 4 & 2550 & 30,25,20,26 &  & 1 & 55 & 5100 & 5197 & 55 & 5197 \\ \hline
103 & 3 & 3536 & 41,32,30 & 2 &  & 71 & 7072 & 8661 & 65 & 7187 \\ \hline
103 & 6 & 1768 & 24,18,14,17,14,16 & 5 &  & 80 & 8840 & 11127 & 72 & 8947 \\
\hline
103 & 6 & 1768 & 24,18,14,17,14,16 & 4 &  & 66 & 7072 & 7515 & 65 & 7187 \\
\hline
103 & 6 & 1768 & 24,18,14,17,14,16 &  & 2 & 70 & 7072 & 8435 & 65 & 7187 \\
\hline
107 & 8 & 1431 & 17,15,10,13,10,12,17,13 & 7 &  & 77 & 10017 & 10409 & 76 &
10179 \\ \hline
107 & 8 & 1431 & 17,15,10,13,10,12,17,13 &  & 3 & 73 & 8586 & 9027 & 71 &
8661 \\ \hline
109 & 3 & 3960 & 42,30,37 & 2 &  & 72 & 7920 & 8947 & 69 & 8291 \\ \hline
109 & 4 & 2970 & 32,26,22,29 & 3 &  & 76 & 8910 & 10179 & 72 & 8947 \\ \hline
109 & 9 & 1320 & 8,14,20,10,13,10,12,10,12 & 8 &  & 78 & 10560 & 10599 & 78
& 10599 \\ \hline
113 & 4 & 3192 & 32,32,24,25 & 3 &  & 80 & 9576 & 11127 & 75 & 9965 \\ \hline
113 & 7 & 1824 & 10,20,14,17,22,16,14 & 6 &  & 80 & 10944 & 11127 & 80 &
11127 \\ \hline
121 & 4 & 3660 & 36,30,25,30 & 3 &  & 86 & 10980 & 13075 & 80 & 11127 \\
\hline
121 & 4 & 3660 & 36,30,25,30 &  & 1 & 66 & 7320 & 7515 & 66 & 7515 \\ \hline
121 & 5 & 2928 & 16,24,28,25,28 & 4 &  & 88 & 11712 & 13491 & 83 & 12041 \\
\hline
125 & 4 & 3906 & 37,32,26,30 & 3 &  & 89 & 11718 & 13557 & 83 & 12041 \\
\hline
125 & 8 & 1953 & 24,16,14,15,13,16,12,15 & 7 &  & 96 & 13671 & 16243 & 90 &
13935 \\ \hline
125 & 8 & 1953 & 24,16,14,15,13,16,12,15 &  & 3 & 93 & 11718 & 15453 & 83 &
12041 \\ \hline
\end{array}
$\newpage

\textbf{Table }3. The existence and nonexistence of cyclic
symmetric
configurations $v_{k},$\\
$k\leq 15,$ $P(k)\leq v_{\delta }(k)\leq v<G(k)\smallskip $

$
\begin{array}{@{}r|@{\,}r|@{\,}r|@{}c|@{}c|@{}r|@{}r}
\hline
&  &  &  &  & E_{c}(k) &  \\
k & P(k) & v_{\delta }(k) & v<G(k)\text{ in literature} & v<G(k)\text{ in
this work} & \leq & G(k) \\ \hline
2\centerdot & 3 & 3 &  &  & 3\centerdot & 3 \\
3\centerdot & 7 & 7 &  &  & 7\centerdot & 7 \\
4\centerdot & 13 & 13 &  &  & 13\centerdot & 13 \\
5\centerdot & 21 & 21 & 21,\overline{\emph{22}} &  & 23\centerdot & 23 \\
6\centerdot & 31 & 31 & 31,\overline{\emph{32}},\overline{\emph{33}},
\overline{\emph{34}^{c}} &  & 35\centerdot & 35 \\
7\centerdot & 43 & 48 & 48\text{-}50 & 49,50 & 48\centerdot & 51 \\
8\centerdot & 57 & 57 & 57,\overline{\emph{58}},\overline{\emph{59}^{c}}
\text{-}\overline{\emph{62}^{c}},63\text{-}68 & 64,67,68 & 63\centerdot & 69
\\
9\centerdot & 73 & 73 & 73,\overline{\emph{74}},\overline{\emph{75}^{c}}
\text{-}\overline{\emph{79}^{c}},80,\overline{\emph{81}^{c}}\text{-}
\overline{\emph{84}^{c}},85\text{-}88 & 86,87 & 85\centerdot & 89 \\
10\phantom{\centerdot} & 91 & 91 & 91,\overline{\emph{92}},107\text{-}110 &
109 & 107\phantom{\centerdot} & 111 \\
11\phantom{\centerdot} & 111 & 120 & 120,133,135\text{-}144 & 137,139,142
\text{-}144 & 135\phantom{\centerdot} & 145 \\
12\phantom{\centerdot} & 133 & 133 & 133,\overline{\emph{134}}
,156,158,159,161\text{-}170 & 158,162\text{-}165,167,169,170 & 161
\phantom{\centerdot} & 171 \\
13\phantom{\centerdot} & 157 & 168 & 168,183,193\text{-}212 & 197,201,203
\text{-}212 & 193\phantom{\centerdot} & 213 \\
14\phantom{\centerdot} & 183 & 183 & 183,\overline{\emph{184}},225\text{-}254
& 226,227,231,233\text{-}254 & 225\phantom{\centerdot} & 255 \\
15\phantom{\centerdot} & 211 & 255 & 255,267\text{-}302 &
\begin{array}{c}
278,282,284,286,287, \\
290\text{-}302
\end{array}
& 267\phantom{\centerdot} & 303 \\ \hline
\end{array}
$\newpage

\textbf{Table 4. }Values of $v$ for which a cyclic symmetric
configuration $ v_{k}$ exists, $16\leq k\leq 41,$ $P(k)\leq
v<G(k)\smallskip $

$\renewcommand{\arraystretch}{0.93}
\begin{array}{@{}r|@{\,}r|@{}c@{}|@{\,}r|@{\,}r}
\hline
&  &  & E_{c}(k) &  \\
k & P(k) & P(k)\leq v<G(k) & \leq & G(k) \\ \hline
16 & 241 & \overline{\emph{241}^{c}}\text{-}\overline{\emph{254}^{c}}
,255_{b},272_{c},273_{a},288_{b},307_{a},\mathbf{318,320}\text{-}\mathbf{
329,331}\text{-}\mathbf{354}\phantom{\overline{\overline{H}}} & 331 & 355 \\
17 & 273 & 273_{a},\overline{\emph{274}},288_{b},307_{a},342_{c},\mathbf{
343,353,}360_{b},\mathbf{357}\text{-}\mathbf{363,}381_{a},\mathbf{365}\text{-}\mathbf{398} & 365 & 399 \\
18 & 307 & 307_{a},342_{c},360_{b},381_{a},\mathbf{401,403,405}\text{-}
\mathbf{407,410,412}\text{-}\mathbf{418,420}\text{-}\mathbf{432} & 420 & 433
\\
19 & 343 & \overline{\emph{344}},360_{b},381_{a},\mathbf{
455,457,464,467,468,470}\text{-}\mathbf{477,479,481}\text{-}\mathbf{492} &
481 & 493 \\
20 & 381 &
\begin{array}{c}
381_{a},\overline{\emph{382}},\mathbf{503,}506_{c},\mathbf{
508,513,516,519,520,525,}528_{b},\mathbf{527}\text{-}\mathbf{530,} \\
\mathbf{532,}553_{a},\mathbf{534}\text{-}\mathbf{566}
\end{array}
& 534 & 567 \\
21 & 421 &
\begin{array}{c}
\overline{\emph{422}},506_{c},528_{b},553_{a},\mathbf{592,597,601,602,606}
\text{-}\mathbf{609,611,}624_{b}, \\
651_{a},\mathbf{614}\text{-}\mathbf{666}
\end{array}
& 614 & 667 \\
22 & 463 & \overline{\emph{463}},\overline{\emph{464}}
,506_{c},528_{b},553_{a},624_{b},\mathbf{640,644,645,}651_{a},\mathbf{649}
\text{-}\mathbf{712} & 649 & 713 \\
23 & 507 &
\begin{array}{c}
\overline{\emph{507}},\overline{\emph{508}},528_{b},553_{a},624_{b},651_{a},
\mathbf{683,696,698,699,702,707,} \\
\mathbf{709}\text{-}\mathbf{711,}728_{b},\mathbf{713}\text{-}\mathbf{744}
\end{array}
& 713 & 745 \\
24 & 553 &
\begin{array}{c}
553_{a},\overline{\emph{554}},624_{b},651_{a},728_{b},\mathbf{739,}757_{a},
\mathbf{759,761,763,765}\text{-}\mathbf{770,} \\
\mathbf{772,}812_{c},840_{b},\mathbf{775}\text{-}\mathbf{850}
\end{array}
& 775 & 851 \\
25 & 601 &
\begin{array}{c}
624_{b},651_{a},728_{b},757_{a},812_{c},\mathbf{837},840_{b},\mathbf{842}
\text{-}\mathbf{844},\mathbf{846}\text{-}\mathbf{854,} \\
\mathbf{856}\text{-}\mathbf{863,}871_{a},930_{c},960_{b},\mathbf{865}\text{\textbf{-}}\mathbf{960}
\end{array}
& 865 & 961 \\
26 & 651 &
\begin{array}{c}
651_{a},\overline{\emph{652}},728_{b},757_{a},812_{c},840_{b},871_{a},
\mathbf{900,905}\text{-}\mathbf{907},\mathbf{910,912,} \\
\mathbf{913,916,917,919}\text{-}\mathbf{921,924,925,929,}930_{c},\mathbf{932}
\text{-}\mathbf{941,}960_{b}, \\
\mathbf{943}\text{-}\mathbf{984}
\end{array}
& 943 & 985 \\
27 & 703 &
\begin{array}{c}
728_{b},757_{a},812_{c},840_{b},871_{a},930_{c},960_{b},\mathbf{971,975,977},
\mathbf{978,} \\
\mathbf{987,991},993_{a},\mathbf{994,997},\mathbf{1000,1001,1003}\text{-}
\mathbf{1006,1008,} \\
\mathbf{1010}\text{-}\mathbf{1015,1017,1019,}1023_{b},1057_{a},\mathbf{1021}
\text{-}\mathbf{1106}
\end{array}
& 1021 & 1107 \\
28 & 757 &
\begin{array}{c}
757_{a},\overline{\emph{758}}
,812_{c},840_{b},871_{a},930_{c},960_{b},993_{a},1023_{b},\mathbf{1045,}
1057_{a}, \\
\mathbf{1063,1067,1070,1074,1075,1077,1079}\text{-}\mathbf{1082,1085}\text{-}
\mathbf{1170}
\end{array}
& 1085 & 1171 \\
29 & 813 &
\begin{array}{c}
\overline{\emph{814}}
,840_{b},871_{a},930_{c},960_{b},993_{a},1023_{b},1057_{a},\mathbf{
1146,1151,1152,} \\
\mathbf{1155}\text{-}\mathbf{1158,1162}\text{-}\mathbf{
1167,1169,1172,1173,1175,1177,} \\
\mathbf{1180}\text{-}\mathbf{1185,1187}\text{-}\mathbf{1246}
\end{array}
& 1187 & 1247 \\
30 & 871 &
\begin{array}{c}
871_{a},\overline{\emph{872}},930_{c},960_{b},993_{a},1023_{b},1057_{a},
\mathbf{1198,1199,1219,} \\
\mathbf{1220,1224,1229,1235,1236,1238,1240,1241,1243,} \\
\mathbf{1248,1249,1251,1253,1255,1258,1261,1264}\text{-}\mathbf{1267,} \\
\mathbf{1269}\text{\textbf{-}}\mathbf{1272,}1332_{c},\mathbf{1274}\text{\textbf{-}}\mathbf{1360}
\end{array}
& 1274 & 1361 \\
31 & 931 &
\begin{array}{c}
\overline{\emph{932}},960_{b},993_{a},1023_{b},1057_{a},\mathbf{1324,1325}
,1332_{c},\mathbf{1341,1344,} \\
\mathbf{1345,1346,1348,1349,}1368_{b},1407_{a},\mathbf{1351}\text{\textbf{-}}
\mathbf{1494}
\end{array}
& 1351 & 1495 \\
32 & 993 &
\begin{array}{c}
993_{a},\overline{\emph{994}},1023_{b},1057_{a},1332_{c},1368_{b},\mathbf{
1383,1393,1401,}1407_{a}, \\
\mathbf{1409,1411,1414,1421,1424,1428,1429,1430,1432,1434,} \\
\mathbf{1438}\text{\textbf{-}}\mathbf{1441,1443}\text{\textbf{-}}\mathbf{
1445,1447}\text{\textbf{-}}\mathbf{1457,1459}\text{\textbf{-}}\mathbf{1568}
\end{array}
& 1459 & 1569 \\
33 & 1057 &
\begin{array}{c}
1057_{a},\overline{\emph{1058}},1332_{c},1368_{b},1407_{a},\mathbf{
1492,1506,1507,1518,1521,} \\
\mathbf{1529,1533,1535,1540,1542,1545,1547}\text{-}\mathbf{1553,1555,} \\
\mathbf{1557}\text{-}\mathbf{1559,1561}\text{-}\mathbf{1563,1565}\text{-}
\mathbf{1567,1569}\text{-}\mathbf{1578,1580}\text{-}\mathbf{1591,} \\
1640_{c},1680_{b},\mathbf{1593}\text{-}\mathbf{1718}
\end{array}
& 1593 & 1719 \\ \hline
\end{array}
$\newpage

\textbf{Table 4} (continue). \textbf{\ }Values of $v$ for which
a cyclic symmetric configuration $v_{k}$ exists, $16\leq k\leq
41,$ $P(k)\leq v<G(k)$

$\renewcommand{\arraystretch}{0.91}
\begin{array}{@{}r|@{}r|@{}c@{}|@{}c|@{}c}
\hline
&  &  & E_{c}(k) &  \\
k & P(k) & P(k)\leq v<G(k) & \leq & G(k) \\ \hline
34 & 1123 &
\begin{array}{c}
\overline{\emph{1123}},\overline{\emph{1124}}
,1332_{c},1368_{b},1407_{a},1640_{c},1680_{b},\mathbf{1699,}1723_{a},\mathbf{
1725,}\phantom{\overline{\overline{H}}} \\
\mathbf{1735,1739,1742,1747,1748,1750,1752,1755}\text{\textbf{-}}\mathbf{
1757,1759,} \\
\mathbf{1761,1765}\text{\textbf{-}}\mathbf{1770,1772,1773,1777,1779}\text{\textbf{-}}\mathbf{1785,}1806_{c},1848_{b}, \\
\mathbf{1787}\text{\textbf{-}}\mathbf{1876}
\end{array}
& 1787 & 1877 \\
35 & 1191 &
\begin{array}{c}
\overline{\emph{1192}},1332_{c},1368_{b},1407_{a},1640_{c},1680_{b},1723_{a},
\mathbf{1781,1783,1788,} \\
\mathbf{1801,1805,}1806_{c},\mathbf{1807,1810,1812,1814,1817,1823,182}5, \\
\mathbf{1826,1829}\text{\textbf{-}}\mathbf{1833,1835}\text{\textbf{-}}
\mathbf{1840,}1848_{b},\mathbf{1842}\text{-}\mathbf{1856},1893_{a}, \\
\mathbf{1859}\text{\textbf{-}}\mathbf{1974}
\end{array}
& 1859 & 1975 \\
36 & 1261 &
\begin{array}{c}
1332_{c},1368_{b},1407_{a},1640_{c},1680_{b},1723_{a},1806_{c},1848_{b},
\mathbf{1855,1860,} \\
\mathbf{1886,}1893_{a},\mathbf{1902,1905,1907,1908,1912,1915}\text{\textbf{-}
}\mathbf{1922,} \\
\mathbf{1925}\text{\textbf{-}}\mathbf{1930,1932,1937}\text{-}\mathbf{
1939,1941}\text{-}\mathbf{1952,1954,1955,} \\
\mathbf{1957}\text{\textbf{-}}\mathbf{1959,1961}\text{\textbf{-}}\mathbf{2010
}
\end{array}
& 1961 & 2011 \\
37 & 1333 &
\begin{array}{c}
\overline{\emph{1334}}
,1368_{b},1407_{a},1640_{c},1680_{b},1723_{a},1806_{c},1848_{b},1893_{a},
\mathbf{1973,} \\
\mathbf{1986,1989,2001,2006}\text{-}\mathbf{2008,2010,2017,2018,2023,2024,}
\\
\mathbf{2028,2031,2033,2036}\text{\textbf{-}}\mathbf{2039,2041}\text{\textbf{-}}\mathbf{2043,2045,2046,2048,} \\
\mathbf{2053,2054}\text{\textbf{-}}\mathbf{2057,2059,2061,2063,2065}\text{\textbf{-}}\mathbf{2074,2076}\text{\textbf{-}}\mathbf{2083,} \\
2162_{c},\mathbf{2085}\text{\textbf{-}}\mathbf{2198}
\end{array}
& 2085 & 2199 \\
38 & 1407 &
\begin{array}{c}
1407_{a},1640_{c},1680_{b},1723_{a},1806_{c},1848_{b},1893_{a},\mathbf{
2059,2061,2073,} \\
\mathbf{2088,2089,2092,2094,2096,2097,2099,2100,2101,2103,} \\
\mathbf{2105,2106,2108,2110,2111,2114}\text{-}\mathbf{2116,2118,2124,} \\
\mathbf{2126}\text{-}\mathbf{2129,2136}\text{-}\mathbf{2139,2142}\text{\textbf{-}}\mathbf{2153,2155}\text{\textbf{-}}\mathbf{2157,2159,2161}, \\
2162_{c},\mathbf{2163,2164,2166}\text{\textbf{-}}\mathbf{2170},\mathbf{
2172,2174}\text{\textbf{-}}\mathbf{2178,}2208_{b},2257_{a}\mathbf{,} \\
\mathbf{2180}\text{-}\mathbf{2292}
\end{array}
& 2180 & 2293 \\
39 & 1483 &
\begin{array}{c}
\overline{\emph{1483}},\overline{\emph{1484}}
,1640_{c},1680_{b},1723_{a},1806_{c},1848_{b},1893_{a},2162_{c},2208_{b}, \\
2257_{a},\mathbf{2265,2278,2281,2287,2293,2294,2297,2300,2302,} \\
\mathbf{2304,2315,2317,2323,2324,2326,2330,2338,2340,2341,} \\
\mathbf{2344}\text{\textbf{-}}\mathbf{2346,2348}\text{\textbf{-}}\mathbf{
2350,2352}\text{\textbf{-}}\mathbf{2354,2358,2361}\text{-}\mathbf{2364,} \\
\mathbf{2366}\text{-}\mathbf{2369,2372}\text{-}\mathbf{2383,2385}\text{-}
\mathbf{2393,}2400_{b},\mathbf{2395}\text{-}\mathbf{2401,}2451_{a}, \\
\mathbf{2403}\text{-}\mathbf{2504}
\end{array}
& 2403 & 2505 \\
40 & 1561 &
\begin{array}{c}
\overline{\emph{1562}}
,1640_{c},1680_{b},1723_{a},1806_{c},1848_{b},1893_{a},2162_{c},2208_{b},2257_{a},
\\
\mathbf{2326,2345,2372,2374,2389,2393,2396,}2400_{b}\mathbf{,2401,2404,} \\
\mathbf{2411,2414,2416,2417,2418,2423,2424,2427,2431,2435,} \\
\mathbf{2436,2438,2440}\text{-}\mathbf{2444,}2451_{a},\mathbf{2449}\text{-}
\mathbf{2453,2455,2459}\text{-}\mathbf{2461,} \\
\mathbf{2464}\text{-}\mathbf{2467,2471,2474}\text{-}\mathbf{2480,2482}\text{-}\mathbf{2484,2486,2487,} \\
\mathbf{2489}\text{-}\mathbf{2522,2524}\text{-}\mathbf{2564}
\end{array}
& 2524 & 2565 \\
41 & 1641 &
\begin{array}{c}
\overline{\emph{1642}}
,1680_{b},1723_{a},1806_{c},1848_{b},1893_{a},2162_{c},2208_{b},2257_{a},
\mathbf{2345,} \\
2400_{b},\mathbf{2449,}2451_{a}\mathbf{,2460,2479,2480,2491,2494,2496,2499,}
\\
\mathbf{2508}\text{-}\mathbf{2511,2513,2516,2518}\text{\textbf{-}}\mathbf{
2521,2524,2525,2528}\text{\textbf{-}}\mathbf{2540,} \\
\mathbf{2542,2544,2546}\text{\textbf{-}}\mathbf{2548,2550}\text{\textbf{-}}
\mathbf{2553,2555}\text{\textbf{-}}\mathbf{2562,2564}\text{\textbf{-}}
\mathbf{2575,} \\
\mathbf{2577}\text{\textbf{-}}\mathbf{2610}
\end{array}
& 2577 & 2611 \\ \hline
\end{array}
$

\newpage

\textbf{Table 5. }Upper bounds on the cyclic existence bound
$E_{c}(k),$ $ 42\leq k\leq 83\smallskip $

$\renewcommand{\arraystretch}{1.0}
\begin{array}{ccc|ccc|ccc|crr}
\hline
k & E_{c}(k) & G(k) & k & E_{c}(k) & G(k) & k & E_{c}(k) & G(k) & k &
E_{c}(k) & G(k) \\ \hline
42 & 2632 & 2795 & 53 & 4463 & 4695 & 64 & 6796 & 7055 & 75 & 9883 & 9965 \\
43 & 2860 & 3015 & 54 & 4513 & 4747 & 65 & 6853 & 7187 & 76 & 10023 & 10179
\\
44 & 2917 & 3193 & 55 & 5195 & 5197 & 66 & 7279 & 7515 & 77 & 10229 & 10409
\\
45 & 3280 & 3375 & 56 & 5341 & 5451 & 67 & 7359 & 7639 & 78 & 10395 & 10599
\\
46 & 3353 & 3407 & 57 & 5501 & 5547 & 68 & 7463 & 7913 & 79 & 10800 & 10817
\\
47 & 3453 & 3609 & 58 & 5551 & 5703 & 69 & 8111 & 8291 & 80 & 10977 & 11127
\\
48 & 3765 & 3775 & 59 & 5612 & 5823 & 70 & 8125 & 8435 & 81 & 11396 & 11435
\\
49 & 3839 & 3917 & 60 & 5687 & 6039 & 71 & 8288 & 8661 & 82 & 11443 & 11629
\\
50 & 3871 & 4189 & 61 & 5994 & 6269 & 72 & 8694 & 8947 & 83 & 11593 & 12041
\\
51 & 4308 & 4381 & 62 & 6150 & 6431 & 73 & 8813 & 9027 &  &  &  \\
52 & 4359 & 4541 & 63 & 6611 & 6783 & 74 & 8965 & 9507 &  &  &  \\ \hline
\end{array}
$\newpage

\textbf{Table 6.} Values of $v$ for which a symmetric
configuration $v_{k}$ (cyclic or non-cyclic) of families of
Section \ref{sec_known} exists, $k\leq 37,$ $P(k)\leq
v<G(k)\smallskip $

$\renewcommand{\arraystretch}{1.0}
\begin{array}{@{}r|r|@{}c|r}
\hline
k & P(k) & P(k)\leq v<G(k)\text{ } & G(k) \\ \hline
8 & 57 & 57_{a},63_{b,e},64_{m}\text{-}68_{m} & 69 \\
9 & 73 & 73_{a},78_{f,g},80_{b,e},81_{m}\text{-}88_{m} & 89 \\
10 & 91 & 91_{a},98_{h},110_{c,d,e,i,m} & 111 \\
11 & 111 & 120_{b,e},121_{m}\text{-}133_{m},143_{m}\text{-}144_{m} & 145 \\
12 & 133 & 133_{a},156_{m}\text{-}170_{m} & 171 \\
13 & 157 & 168_{b,e},169_{m}\text{-}183_{m},189_{f},208_{m}\text{-}212_{m} &
213 \\
14 & 183 & 183_{a},210_{f},224_{m}\text{-}254_{m} & 255 \\
15 & 211 & 231_{f},240_{m}\text{-}302_{m} & 303 \\
16 & 241 & 252_{f,g},255_{b,e},256_{m}\text{-}321_{m},323_{m}\text{-}354_{m}
& 355 \\
17 & 273 & 273_{a},288_{b,e},289_{m}\text{-}307_{m},323_{m}\text{-}
381_{m},391_{m}\text{-}398_{m} & 399 \\
18 & 307 & 307_{a},342_{m}\text{-}381_{m},403_{f},414_{m}\text{-}432_{m} &
433 \\
19 & 343 & 360_{b,e},361_{m}\text{-}381_{m},434_{f},437_{m}\text{-}
457_{m},460_{m}\text{-}492_{m} & 493 \\
20 & 381 & 381_{a},460_{m}\text{-}481_{m},483_{m}\text{-}\ 566_{m} & 567 \\
21 & 421 & 483_{m}\text{-}666_{m} & 667 \\
22 & 463 & 506_{m}\text{-}573_{m},575_{m}\text{-}712_{m} & 713 \\
23 & 507 & 528_{b,e},529_{m}\text{-}553_{m},558_{f},575_{m}\text{-}744_{m} &
745 \\
24 & 553 & 553_{a},589_{f},600_{m}\text{-}673_{m},675_{m}\text{-}850_{m} &
851 \\
25 & 601 & 620_{f,g},624_{b,e},625_{m}\text{-}651_{m},675_{m}\text{-}960_{m}
& 961 \\
26 & 651 & 651_{a},702_{m}\text{-}781_{m},783_{m}\text{-}984_{m} & 985 \\
27 & 703 & 728_{b,e},729_{m}\text{-}757_{m},783_{m}\text{-}1065_{m},1073_{m}
\text{-}1103_{m} & 1107 \\
28 & 757 & 757_{a},812_{m}\text{-}1065_{m},1073_{m}\text{-}1103_{m},1110_{m}
\text{-}1141_{m},1147_{m}\text{-}1170_{m} & 1171 \\
29 & 813 &
\begin{array}{@{}c}
840_{b,e},841_{m}\text{-}871_{m},899_{m}\text{-}1057_{m},1073_{m}\text{-}
1103_{m},1110_{m}\text{-}1141_{m}, \\
1147_{m}\text{-}1179_{m},1184_{m}\text{-}1219_{m},1221_{m}\text{-}1246_{m}
\end{array}
& 1247 \\
30 & 871 &
\begin{array}{c}
871_{a},930_{m}\text{-}1057_{m},1110_{m}\text{-}1141_{m},1147_{m}\text{-}
1179_{m},1184_{m}\text{-}1217_{m}, \\
1221_{m}\text{-}1360_{m}
\end{array}
& 1361 \\
31 & 931 &
\begin{array}{c}
960_{b,e},961_{m}\text{-}1057_{m},1147_{m}\text{-}1179_{m},1184_{m}\text{-}
1217_{m},1221_{m}\text{-}1255_{m}, \\
1258_{m}\text{-}1494_{m}
\end{array}
& 1495 \\
32 & 993 &
\begin{array}{c}
993_{a},1023_{b,e},1024_{m}\text{-}1057_{m},1184_{m}\text{-}1217_{m},1221_{m}
\text{-}1255_{m}, \\
1258_{m}\text{-}1293_{m},1295_{m}\text{-}1568_{m}
\end{array}
& 1569 \\
33 & 1057 & 1057_{a},1221_{m}\text{-}1255_{m},1258_{m}\text{-}
1293_{m},1295_{m}\text{-}1718_{m} & 1719 \\
34 & 1123 & 1258_{m}\text{-}1293_{m},1295_{m}\text{-}1429_{m},1435_{m}\text{-}1876_{m} & 1877 \\
35 & 1191 & 1295_{m}\text{-}1407_{m},1435_{m}\text{-}1471_{m},1476_{m}\text{-}1974_{m} & 1975 \\
36 & 1261 & 1332_{m}\text{-}1407_{m},1476_{m}\text{-}1513_{m},1517_{m}\text{-}2010_{m} & 2011 \\
37 & 1333 & 1368_{b,e},1369_{m}\text{-}1407_{m},1517_{m}\text{-}
1555_{m},1558_{m}\text{-}2198_{m} & 2199 \\ \hline
\end{array}
$\newpage

\textbf{Table 7.} Parameters of new symmetric configurations
$v_{k}$ (cyclic and non-cyclic) from Sections
\ref{sec_doub-circ} and \ref{sec_admitExten}
\smallskip

$\renewcommand{\arraystretch}{1.0}
\begin{array}{@{}r|@{}r|@{}c|@{}r}
\hline
k & P(k) & P(k)\leq v<G(k)\text{ } & G(k) \\ \hline
8 & 57 & 63_{r}\text{-}68_{r} & 69 \\
9 & 73 & 80_{r}\text{-}88_{r} & 89 \\
10 & 91 & 110_{r,s,t} & 111 \\
11 & 111 & 120_{r}\text{-}133_{r},143_{s},144_{r,s,t} & 145 \\
12 & 133 & 156_{s}\text{-}169_{s},156_{r,t}\text{-}170_{r,t} & 171 \\
13 & 157 & 168_{r}\text{-}183_{r},210_{r}\text{-}212_{r} & 213 \\
14 & 183 & 225_{r}\text{-}254_{r},238_{s}\text{-}253_{s},240_{t}\text{-}
254_{t} & 255 \\
15 & 211 & 240_{r}\text{-}302_{r},255_{s}\text{-}301_{s},256_{t}\text{-}
302_{t} & 303 \\
16 & 241 & 255_{r}\text{-}354_{r},272_{s}\text{-}289_{s},272_{t}\text{-}
290_{t}\text{,}304_{s}\text{-}321_{s},306_{t}\text{-}354_{t},323_{s}\text{-}
354_{s} & 355 \\
17 & 273 & 288_{r}\text{-}307_{r},323_{s}\text{-}361_{s},324_{t}\text{-}
362_{t},324_{r}\text{-}381_{r},391_{s}\text{-}398_{s},396_{r,t}\text{-}
398_{r,t} & 399 \\
18 & 307 & 342_{s}\text{-}361_{s},342_{t}\text{-}362_{t},342_{r}\text{-}
381_{r},414_{s}\text{-}432_{s},418_{r,t}\text{-}432_{r,t} & 433 \\
19 & 343 & 360_{r}\text{-}381_{r},437_{s}\text{-}457_{s},440_{r,t}\text{-}
492_{r,t},460_{s}\text{-}481_{s},483_{s}\text{-}492_{s} & 493 \\
20 & 381 & 460_{s}\text{-}481_{s},462_{r,t}\text{-}566_{r,t},483_{s}\text{-}
529_{s} & 567 \\
21 & 421 &
\begin{array}{c}
483_{s}\text{-}529_{s},484_{t}\text{-}530_{t},484_{r}\text{-}666_{r},609_{s}
\text{-}631_{s},616_{t}\text{-}639_{t},638_{s}\text{-}666_{s},640_{w}, \\
644_{t}\text{-}666_{t},651_{u}\text{-}666_{u}
\end{array}
& 667 \\
22 & 463 &
\begin{array}{c}
506_{s}\text{-}529_{s},506_{t}\text{-}530_{t},506_{r}\text{-}712_{r},638_{s}
\text{-}661_{s},640_{w},644_{t}\text{-}668_{t},667_{s}\text{-}712_{s}, \\
672_{t}\text{-}712_{t}
\end{array}
& 713 \\
23 & 507 & 528_{r}\text{-}553_{r},576_{r}\text{-}744_{r},667_{s}\text{-}
691_{s},672_{t}\text{-}697_{t},696_{s}\text{-}744_{s},700_{t}\text{-}744_{t}
& 745 \\
24 & 553 & 600_{r}\text{-}850_{r},696_{s}\text{-}721_{s},700_{t}\text{-}
726_{t},725_{s}\text{-}850_{s},728_{t}\text{-}850_{t} & 851 \\
25 & 601 &
\begin{array}{c}
624_{p},624_{r}\text{-}651_{r}\text{,}676_{r}\text{-}960_{r},725_{s}\text{-}
751_{s},728_{t}\text{-}960_{t},754_{s}\text{-}865_{s},868_{s}\text{-}897_{s},
\\
899_{s}\text{-}960_{s},906_{v},912_{w},938_{v}
\end{array}
& 961 \\
26 & 651 & 702_{r}\text{-}984_{r},754_{s}\text{-}781_{s},756_{t}\text{-}
984_{t},783_{s}\text{-}865_{s},868_{s}\text{-}897_{s},899_{s}\text{-}984_{s}
& 985 \\
27 & 703 &
\begin{array}{c}
728_{r}\text{-}757_{r},783_{s}\text{-}865_{s},784_{t}\text{-}962_{t},784_{r}
\text{-}1074_{r},868_{s}\text{-}897_{s},899_{s}\text{-}961_{s}, \\
999_{s}\text{-}1027_{s},1008_{t}\text{-}1037_{t},1036_{s}\text{-}
1065_{s},1044_{t}\text{-}1074_{t},1073_{s}\text{-}1103_{s}, \\
1080_{r,t}\text{-}1106_{r,t}
\end{array}
& 1107 \\
28 & 757 &
\begin{array}{c}
812_{s}\text{-}841_{s},812_{t}\text{-}842_{t},812_{r}\text{-}1074_{r},868_{s}
\text{-}897_{s},899_{s}\text{-}961_{s},870_{t}\text{-}962_{t}, \\
1036_{s}\text{-}1065_{s},1044_{t}\text{-}1074_{t},1073_{s}\text{-}
1103_{s},1080_{t}\text{-}1111_{t},1080_{r}\text{-}1170_{r}, \\
1110_{s}\text{-}1141_{s},1116_{t}\text{-}1148_{t},1147_{s}\text{-}
1170_{s},1152_{t}\text{-}1170_{t}
\end{array}
& 1171 \\
29 & 813 &
\begin{array}{c}
840_{r}\text{-}871_{r},899_{s}\text{-}961_{s},900_{t}\text{-}962_{t},900_{r}
\text{-}1057_{r},1073_{s}\text{-}1103_{s},1080_{r,t}\text{-}1111_{r,t}, \\
1110_{s}\text{-}1141_{s},1116_{r,t}\text{-}1148_{r,t},1147_{s}\text{-}
1179_{s},1152_{r,t}\text{-}1185_{r,t},1184_{s}\text{-}1219_{s}, \\
1188_{r,t}\text{-}1246_{r,t},1221_{s}\text{-}1246_{s}
\end{array}
& 1247 \\
30 & 871 &
\begin{array}{c}
930_{s}\text{-}961_{s},930_{t}\text{-}962_{t},930_{r}\text{-}
1057_{r},1110_{s}\text{-}1141_{s},1116_{r,t}\text{-}1148_{r,t}, \\
1147_{s}\text{-}1179_{s},1152_{r,t}\text{-}1185_{r,t},1184_{s}\text{-}
1217_{s},1188_{r,t}\text{-}1222_{r,t},1221_{s}\text{-}1360_{s}, \\
1224_{r,t}\text{-}1360_{r,t},1262_{v}
\end{array}
& 1361 \\ \hline
\end{array}
$\newpage

\textbf{Table 7 }(continue\textbf{).} Parameters of new
symmetric configurations $v_{k}$ (cyclic and non-cyclic) from
Sections \ref {sec_doub-circ} and
\ref{sec_admitExten}\smallskip

$\renewcommand{\arraystretch}{1.0}
\begin{array}{@{}r|@{}r|@{}c|@{}r}
\hline
k & P(k) & P(k)\leq v<G(k)\text{ } & G(k) \\ \hline
31 & 931 &
\begin{array}{c}
960_{r}\text{-}1057_{r},1147_{s}\text{-}1179_{s},1152_{r,t}\text{-}
1185_{r,t},1184_{s}\text{-}1217_{s},1188_{r,t}\text{-}1222_{r,t}, \\
1221_{s}\text{-}1255_{s},1224_{r,t}\text{-}1494_{r,t},1258_{s}\text{-}
1494_{s},1262_{v}
\end{array}
& 1495 \\
32 & 993 &
\begin{array}{c}
1023_{r}\text{-}1057_{r},1184_{s}\text{-}1217_{s},1188_{r,t}\text{-}
1222_{r,t},1221_{s}\text{-}1255_{s},1224_{r,t}\text{-}1568_{r,t}, \\
1258_{s}\text{-}1293_{s},1295_{s}\text{-}1568_{s}
\end{array}
& 1569 \\
33 & 1057 &
\begin{array}{c}
1221_{s}\text{-}1255_{s},1224_{t}\text{-}1395_{t},1224_{r}\text{-}
1718_{r},1258_{s}\text{-}1293_{s},1295_{s}\text{-}1387_{s}, \\
1394_{s}\text{-}1718_{s},1400_{t}\text{-}1718_{t},1634_{v}
\end{array}
& 1719 \\
34 & 1123 &
\begin{array}{c}
1258_{s}\text{-}1293_{s},1260_{t}\text{-}1370_{t},1260_{r}\text{-}
1436_{r},1295_{s}\text{-}1369_{s},1394_{s}\text{-}1429_{s}, \\
1400_{t}\text{-}1436_{t},1435_{s}\text{-}1876_{s},1440_{r,t}\text{-}
1876_{r,t},1634_{v},1800_{p,w}
\end{array}
& 1877 \\
35 & 1191 &
\begin{array}{c}
1295_{s}\text{-}1369_{s},1296_{t}\text{-}1370_{t},1296_{r}\text{-}
1407_{r},1435_{s}\text{-}1471_{s},1440_{r,t}\text{-}1477_{r,t}, \\
1476_{s}\text{-}1974_{s},1480_{r,t}\text{-}1974_{r,t},1800_{p}
\end{array}
& 1975 \\
36 & 1261 &
\begin{array}{c}
1332_{s}\text{-}1369_{s},1332_{t}\text{-}1370_{t},1332_{r}\text{-}
1407_{r},1476_{s}\text{-}1513_{s},1480_{r,t}\text{-}1518_{r,t}, \\
1517_{s}\text{-}1873_{s},1520_{r,t}\text{-}2010_{r,t},1880_{s}\text{-}
2010_{s},2000_{w}
\end{array}
& 2011 \\
37 & 1333 &
\begin{array}{c}
1368_{r}\text{-}1407_{r},1517_{s}\text{-}1555_{s},1520_{t}\text{-}
1881_{t},1520_{r}\text{-}2198_{r},1558_{s}\text{-}1717_{s}, \\
1886_{t}\text{-}1928_{t},1720_{s}\text{-}1873_{s},1880_{s}\text{-}
1921_{s},1927_{s}\text{-}2065_{s},1932_{t}\text{-}2022_{t}, \\
2024_{t}\text{-}2068_{t},2067_{s}\text{-}2113_{s},2070_{t}\text{-}
2198_{t},2106_{w},2109_{u}\text{-}2147_{u}, \\
2115_{s}\text{-}2198_{s},2166_{u}\text{-}2198_{u}
\end{array}
& 2199 \\
38 & 1407 & 1560_{r}\text{-}2292_{r},2166_{u}\text{-}2205_{u},2223_{u}\text{-}2263_{u},2280_{u}\text{-}2292_{u} & 2293 \\
39 & 1483 &
\begin{array}{c}
1600_{r}\text{-}2504_{r},2223_{u}\text{-}2263_{u},2280_{u}\text{-}
2321_{u},2337_{u}\text{-}2379_{u},2394_{u}\text{-}2437_{u}, \\
2400_{p},2451_{u}\text{-}2495_{u},2480_{w}
\end{array}
& 2505 \\
40 & 1561 &
\begin{array}{c}
1640_{r}\text{-}1928_{r},1932_{r}\text{-}2564_{r},2280_{u}\text{-}
2321_{u},2337_{u}\text{-}2379_{u},2394_{u}\text{-}2437_{u}, \\
2400_{p},2451_{u}\text{-}2495_{u},2480_{w},2508_{u}\text{-}2553_{u},2522_{v}
\end{array}
& 2565 \\
41 & 1641 &
\begin{array}{c}
1680_{r}\text{-}1723_{r},1764_{r}\text{-}1893_{r},1932_{r}\text{-}
1975_{r},1978_{r}\text{-}2492_{r},2400_{p},2480_{w}, \\
2494_{r}\text{-}2610_{r},2522_{v}
\end{array}
& 2611 \\
46 & 2071 & 2400_{p},2774_{v} & 3407 \\
49 & 2353 & 2400_{p} & 3917 \\
56 & 3081 & 4745_{n,v},4920_{w},5200_{w} & 5451 \\ \hline
\end{array}
$\newpage

\textbf{Table 8. }The current known parameters of symmetric
configurations $ v_{k}$ (cyclic and non-cyclic), an integrated
table\smallskip

$\renewcommand{\arraystretch}{1.0}
\begin{array}{r|r|c|r|r|r}
\hline
k & P(k) & P(k)\leq v<G(k) & E(k)\leq & G(k) & \text{filling} \\ \hline
3\centerdot & 7 & \emph{7} & 7\centerdot & 7 &100\%  \\
4\centerdot & 13 & 13 & 13\centerdot & 13 & 100\% \\
5\centerdot & 21 & 21,\overline{\emph{22}} & 23\centerdot & 23 & 100\% \\
6\centerdot & 31 & 31,\overline{\emph{32}},\overline{\emph{33}},34 &
35\centerdot & 35 & 100\% \\
7\phantom{\centerdot} & 43 & \overline{\emph{43}},\overline{\emph{44}},45,48
\text{-}50 & 48\phantom{\centerdot} & 51 &75\%  \\
8\phantom{\centerdot} & 57 & 57,\overline{\emph{58}},63\text{-}68 & 63
\phantom{\centerdot} & 69 &67\%  \\
9\phantom{\centerdot} & 73 & 73,\overline{\emph{74}},78,80\text{-}88 & 80
\phantom{\centerdot} & 89 &75\%  \\
10\phantom{\centerdot} & 91 & 91,\overline{\emph{92}},98,107\text{-}110 & 107
\phantom{\centerdot} & 111 & 35\% \\
11\phantom{\centerdot} & 111 & \overline{\emph{111}},\overline{\emph{112}}
,120\text{-}133,135\text{-}144 & 135\phantom{\centerdot} & 145 & 76\% \\
12\phantom{\centerdot} & 133 & 133,\overline{\emph{134}},135,156\text{-}170
& 156\phantom{\centerdot} & 171 & 47 \% \\
13\phantom{\centerdot} & 157 & \overline{\emph{158}},168\text{-}183,189,193
\text{-}212 & 193\phantom{\centerdot} & 213 & 68\% \\
14\phantom{\centerdot} & 183 & 183,\overline{\emph{184}},210,224\text{-}254
& 224\phantom{\centerdot} & 255 & 47\% \\
15\phantom{\centerdot} & 211 & \overline{\emph{211}},\overline{\emph{212}}
,231,240\text{-}302 & 240\phantom{\centerdot} & 303 & 72\% \\
16\phantom{\centerdot} & 241 & 252,255\text{-}354 & 255\phantom{\centerdot}
& 355 & 89\% \\
17\phantom{\centerdot} & 273 & 273,\overline{\emph{274}},288\text{-}307,323
\text{-}398 & 323\phantom{\centerdot} & 399 &78\%  \\
18\phantom{\centerdot} & 307 & 307,342\text{-}381,401,403,405\text{-}
407,410,412\text{-}432 & 412\phantom{\centerdot} & 433 &54\%  \\
19\phantom{\centerdot} & 343 & \overline{\emph{344}},360\text{-}381,434,437
\text{-}492 & 437\phantom{\centerdot} & 493 &53\%  \\
20\phantom{\centerdot} & 381 & 381,\overline{\emph{382}},460\text{-}566 & 460
\phantom{\centerdot} & 567 &59\%  \\
21\phantom{\centerdot} & 421 & \overline{\emph{422}},483\text{-}666 & 483
\phantom{\centerdot} & 667 & 75\% \\
22\phantom{\centerdot} & 463 & \overline{\emph{463}},\overline{\emph{464}}
,506\text{-}712 & 506\phantom{\centerdot} & 713 & 84\% \\
23\phantom{\centerdot} & 507 & \overline{\emph{507}},\overline{\emph{508}}
,528\text{-}553,558,575\text{-}744 & 575\phantom{\centerdot} & 745 & 84\% \\
24\phantom{\centerdot} & 553 & 553,\overline{\emph{554}},589,600\text{-}850
& 600\phantom{\centerdot} & 851 & 85\% \\
25\phantom{\centerdot} & 601 & 620,624\text{-}651,675\text{-}960 & 675
\phantom{\centerdot} & 961 &88\%  \\
26\phantom{\centerdot} & 651 & 651,\overline{\emph{652}},702\text{-}984 & 702
\phantom{\centerdot} & 985 &  85\%\\
27\phantom{\centerdot} & 703 & 728\text{-}757,783\text{-}1106 & 783
\phantom{\centerdot} & 1107 & 88 \% \\
28\phantom{\centerdot} & 757 & 757,\overline{\emph{758}},812\text{-}1170 &
812\phantom{\centerdot} & 1171 &87\%  \\
29\phantom{\centerdot} & 813 & \overline{\emph{814}},840\text{-}871,899\text{-}1057,1073\text{-}1246 & 1073\phantom{\centerdot} & 1247 &84\%  \\
30\phantom{\centerdot} & 871 & 871,\overline{\emph{872}},930\text{-}1057,1110
\text{-}1360 & 1110\phantom{\centerdot} & 1361 &78\%  \\
31\phantom{\centerdot} & 931 & \overline{\emph{931}},\overline{\emph{932}}
,960\text{-}1057,1147\text{-}1494 & 1147\phantom{\centerdot} & 1495 & 68 \% \\
32\phantom{\centerdot} & 993 & 993,\overline{\emph{994}},1023\text{-}
1057,1184\text{-}1568 & 1184\phantom{\centerdot} & 1569 &73\% \\
33\phantom{\centerdot} & 1057 & 1057,\overline{\emph{1058}},1221\text{-}1718
& 1221\phantom{\centerdot} & 1719 & 76\% \\
34\phantom{\centerdot} & 1123 & \overline{\emph{1123}},\overline{\emph{1124}}
,1258\text{-}1876 & 1258\phantom{\centerdot} & 1877 &82\%  \\
35\phantom{\centerdot} & 1191 & \overline{\emph{1192}},1295\text{-}1407,1435
\text{-}1974 & 1435\phantom{\centerdot} & 1975 & 83\% \\
36\phantom{\centerdot} & 1261 & 1332\text{-}1407,1476\text{-}2010 & 1476
\phantom{\centerdot} & 2011 &81\%  \\
37\phantom{\centerdot} & 1333 & \overline{\emph{1334}},1368\text{-}1407,1517
\text{-}2198 & 1517\phantom{\centerdot} & 2199 & 84\% \\ \hline
\end{array}
$
\end{center}


\begin{thebibliography}{99}
\bibitem{AFLN-graphs} M. Abreu, M. Funk, D. Labbate, and V.
    Napolitano, On (minimal) regular graphs of girth~6,
    Australas  J  Combin  35 (2006), 119--132.

\bibitem{AFLN-ConfigGraphs} M. Abreu, M. Funk, D. Labbate, and
    V. Napolitano, Configuration graphs of neighbourhood
    geometries, Contrib. Discrete Math   3 (2008),
    109--122.

\bibitem{AfDaZ} V. B. Afanassiev, A. A. Davydov, and V. V.
    Zyablov,$\ $Low density concatenated codes with
    Reed-Solomon component codes, In: Proc. XI Int. Symp. on
    Problems of Redundancy in Inf. and Control Syst.,
    St.-Petersburg, Russia, 2007, pp. 47--51.
    http://k36.org/redundancy2007

\bibitem{AfDaZ-InfProc} V. B. Afanassiev, A. A. Davydov, and V.
    V. Zyablov, Low density parity check codes on bipartite
    graphs with Reed-Solomon constituent codes, Information
    Processes (Electronic Journal) 9, no. 4,
    (2009), 301--331. http://www.jip.ru/2009/301-331-2009.pdf

\bibitem{Baker} R. D. Baker, An elliptic semiplane, J  Combin
    Theory A
 25 (1978), 193--195.

\bibitem{BargZemor} A. Barg and G. Z\'{e}mor, Distances
    properties of expander codes, IEEE Trans Inform Theory
    52 (2006), 78--90.

\bibitem{Boben-v3} M. Boben, Irreducible $(v3)$ configurations
    and graphs, Discrete Math  307 (2007), 331--344.

\bibitem{Bose} R. C. Bose, An affine analogue of Singer's
    theorem, J  Ind Math  Society 6 (1942), 1--15.

\bibitem{CarsDinStef-Reduc-n3} H. G. Carstens, T. Dinski, and
    E. Steffen, Reduction of symmetric configurations $n_{3}$,
    Discrete Applied Math  99 (2000), 401--411.

\bibitem{BrasAmorosSemigroup} M. Bras-Amor\'{o}s and K.
    Stockes, The semigroup of combinatorial configurations,
    Semigroup forum, to appear, DOI: 10.1007/s00233-011-9343-5

\bibitem{DC2} J.\ Coykendall and J.\ Dover, Sets with few
    intersection numbers from Singer subgroup orbits, Europ J
    Combin 22 (2001), 455--464.

\bibitem{DGMP-ACCT2008} A. A. Davydov, M. Giulietti, S.
    Marcugini, and F. Pambianco,  Symmetric configurations
    for bipartite-graph codes, In: Proc. XI Int. Workshop
    Algebraic Comb. Coding Theory, ACCT2008, Pamporovo,
    Bulgaria, 2008, pp. 63--69.
    http://www.moi.math.bas.bg/acct2008/b11.pdf

\bibitem{DGMP-Petersb2009} A. A. Davydov, M. Giulietti, S.
    Marcugini, and F. Pambianco, On the spectrum of possible
    parameters of symmetric configurations, In: Proc. XII Int.
    Symp. on Problems of Redundancy in Inform. and Control
    Systems, Saint-Petersburg, Russia, 2009, pp. 59--64.
    http://k36.org/redundancy2009

\bibitem{DGMP-GraphCodes} A. A. Davydov, M. Giulietti, S.
    Marcugini, and F. Pambianco, Some Combinatorial Aspects of
    Constructing Bipartite-Graph Codes, Graph and Combin, to
    appear, DOI 10.1007/s00373-011-1103-5.

\bibitem{DomingoAmorosPeertoPeer} J. Domingo-Ferrer, M.
    Bras-Amor\'{o}s, Q. Wu, and J. Manj\'{o}n, User-private
    information retrieval based on a peer-to-peer community,
    Data Knowl Eng 68 (2009), 1237--1252.

\bibitem{Dimit} A. Dimitromanolakis, Analysis of the Golomb
    Ruler and the Sidon Set Problems, and Determination of
    Large, Near-Optimal Golomb Rulers, Depart. Electronic
    Comput. Eng. Techn. University of Crete, 2002.
    http://www.cs.toronto.edu/ $\sim $apostol/golomb/main.pdf

\bibitem{Draka} K. Drakakis, A review of the available
    construction methods for Golomb rulers, Advanc Math
    Commun 3 (2009), 235--250.

\bibitem{Funk2008} M. Funk, Cyclic Difference Sets of Positive
    Deficiency, Bull Inst Combin Appl 53
    (2008), 47--56.

\bibitem{FunkLabNap} M. Funk, D. Labbate, and V. Napolitano,
    Tactical (de-)compositions of symmetric configurations,
     Discrete Math 309 (2009), 741--747.

\bibitem{GabidISIT} E. Gabidulin, A. Moinian, and B. Honary,
    Generalized construction of quasi-cyclic regular LDPC codes
    based on permutation matrices,  In: Proc. Int. Symp. Inf.
    Theory 2006, ISIT 2006, Seattle, USA, 2006, pp. 679--683.

\bibitem{GH} A.\ G\'{a}cs and T.\ H\'{e}ger, On geometric
    constructions of $(k,g)$-graphs, Contrib  Discr
    Math  3 (2008), 63--80.

\bibitem{WolfMathWorld-GR} Golomb Ruler. In Wolfram MathWorld.
    Available at http://mathworld.wolfram. com/GolombRuler.html

\bibitem{GrahSloan} R. L. Graham and N. J. A. Sloane, On
    additive bases and Harmonious Graphs, Siam J Algeb
    Discrete Methods 1 (1980), 382--404.

\bibitem{Gropp-nk} H. Gropp, On the existence and non-existence
    of configurations $n_{k}$, J Combin Inform System
    Sci 15 (1990), 34--48.

\bibitem{Gropp-Chemic} H. Gropp, Configurations, regular graphs
    and chemical compounds, J Mathematical Chemistry 11 (1992),
    145-153.

\bibitem{Gropp-nonsim} H. Gropp, Non-symmetric configurations
    with deficiencies 1 and 2, in Combinatorics 90, Ann
    Discrete Math  52 (1992), 227--239.

\bibitem{Gropp-ConfGraph} H. Gropp, Configurations and graps --
    II, Discrete Math 164 (1997), 155--163.

\bibitem{Gropp-ConfGeomCombin} H. Gropp, Configurations between
    geometry and combinatorics, Discr Appl Math 138
    (2004), 79--88.

\bibitem{Gropp-Handb} H. Gropp, ``Configurations,'' The CRC
    Handbook of Combinatorial Designs, 2nd edn, Chapter VI.7, C. J.
    Colbourn and J.~Dinitz (Editors), Chapman \& Hall, CRC,
    New York,  2006,  pp.
    353--355.

\bibitem{Grunbaum} B. Gr\"{u}nbaum, Configurations of points
    and line, Gradute studies in mathematics 103,
    American Mathematical Society, Providence, 2009.

\bibitem{OstergModulSidon} H. Haanp\"{a}\"{a}, A. Huima, and
    P. R. J. \"{O} sterg\aa rd, Sets in $Z_{n}$ with distinct
    sums of pairs, Discrete Appl Math 138 (2004),
    99--106.

\bibitem{Hirs} J. W. P. Hirschfeld, Projective Geometries over
    Finite Fields, 2nd edn, Oxford University Press, Oxford,
    UK,  1998.

\bibitem{KaskiOst} P. Kaski and P. R. J. \"{O}sterg\aa rd,
    There
    exists no symmetric configuration with 33 points and line
    size 6, Australas J Combin 38 (2007), 273--277.

\bibitem{KaskOst112-11} P. Kaski and P.R.J. \"{O}sterg\aa rd,
    There are exactly five biplanes with k = 11, J Combin
    Des 16 (2007), 117-127.

\bibitem{QC Encoder} Z.-W. Li, L. Chen, L. Zeng, S. Lin, and W.
    H. Fong,  Efficient encoding of quasi-cyclic
    low-density parity-check codes, IEEE Trans Commun
    54 (2006), 71--81.

\bibitem{Krcadinac} V. Kr\v{c}adinac, Construction and
    classification of finite structures by computer, PhD
    thesis, University of Zagreb, 2004 (in Croatian).

\bibitem{LingDTSAP} A. C. H. Ling, Difference Triangle Sets
    from affine planes, IEEE Trans Inform Theory 48
    (2002), 2399--2401.

\bibitem{Lipman} M. J. Lipman, The Existence of small
    tactical configurations,\ in: Graphs and Combinatorics,
    Springer Lecture Notes in Mathematics 406,
    Springer-Verlag, 1974, pp. 319--324.

\bibitem{Martinetti} Martinetti V. Sulle configurazioni piane
    $\mu_{3}$. Annali di matematica pura ed applicata (2).
    15 (1887-88), 1--26.

\bibitem{MePaWolk} N. S. Mendelsohn, R. Padmanabhan, and B.
    Wolk, Planar projective configurations I, Note di
    Matematica 7 (1987), 91--112.
    http://siba2.unile.it/ese/issues/1/13/\newline
    Notematv7n1p91.pdf

\bibitem{MilFos2008} N. Miladinovi\'{c} and M. P. C. Fossorier,
    Generalized LDPC codes and generalized stopping sets,
    IEEE Trans Commun 56 (2008) 201--212.

\bibitem{BibliogrSidon} K. O'Bryant, A complete annotated
    bibliography of work related to Sidon sequences, Electronic
    J Combin 11 (2004), 39.

\bibitem{Pepe} V. Pepe, LDPC codes from the Hermitian curve,
    Des Codes Crypt 42 (2007), 303--315.

\bibitem{Ruzsa} I. Z. Ruzsa, Solving a linear equation in a set
    of integers I, Acta Arithmetica 65 (1993),
    259--282.

\bibitem{Shearer-Handb} J. Shearer, ``Difference Triangle
    Sets,''
    The CRC Handbook of Combinatorial Designs, Chapter VI.19, 2nd
    edn, C. J. Colbourn and J. Dinitz (Editors), Chapman \& Hall, CRC,
    New York,  2006, pp. 436--440.

\bibitem{ShearerWebShortest} J. Shearer, Table of lengths of shortest known
Golomb rulers, http://www.research. ibm.com/people/s/shearer/grtab.html

\bibitem{ShearerWebModulGR} J. Shearer, Modular Golomb rulers,
    http://www.research.ibm.com/people/s/shearer /mgrule.html

\bibitem{ShearerReport} J. B. Shearer, Difference Triangle
    Sets  constructions, IBM Research report,
    RC24623(W0808-045), 2008.

\bibitem{Singer} J. Singer, A theorem in finite projective
    geometry and some applications to number theory, Trans
    American Math Soc 43  (1938), 377--385.

\bibitem{StokesPeertoPeer} Stokes, K., Bras-Amor\'{o}s, M.:
    Optimal configurations for peer-to-peer user-private
    information retrieval. Comput Math Applications
    59 (2010), 1568--1577.

\bibitem{Swanson} C. N. Swanson, Planar cyclic difference
    packings, J Combin Des 8 (2000),
    426--434.
\end{thebibliography}
\end{document}